\documentclass[12pt,thmsa]{amsart}

\let\oldtocsection=\tocsection

\let\oldtocsubsection=\tocsubsection

\renewcommand{\tocsection}[2]{\hspace{0em}\oldtocsection{#1}{#2}}
\renewcommand{\tocsubsection}[2]{\hspace{2em}\oldtocsubsection{#1}{#2}}
\usepackage{setspace}

\usepackage{amssymb, euscript,  mathrsfs}
\usepackage[dvips]{graphics}
\usepackage[title]{appendix}

\input{amssym.def}

\input{amssym.tex}

\usepackage{lineno}

\usepackage{tikz}
\usepackage[utf8]{inputenc}

\usetikzlibrary{matrix,arrows,decorations.pathmorphing}

\usepackage[all,cmtip]{xy}

\usepackage{amsmath}
\let\oldAA\AA
\renewcommand{\AA}{\text{\normalfont\oldAA}}

\def\Cal{\mathcal}

\def\A{{\Cal A}}

\def\C{{\Cal C}}

\def\R{{\Cal R}}

\def\P{{\Cal P}}

\def\S{{\Cal S}}
\def\F{{\Cal F}}

\def\T{{\Cal T}}

\def\bbr{{\Bbb R}}

\def\bbn{{\Bbb N}}

\def\bbc{{\Bbb C}}
\def\bbz{{\Bbb Z}}
\def\bbd{{\Bbb D}}

\def\bbs{{\Bbb S}}

\def\sgn{{\hbox{\rm sgn}}}

\def\const{{\hbox{\rm const}}}

\def\rn{\bbr^n}
\def\sn{S^{n-1}}

\def\part{\partial}
\def\intl{\int\limits}
\def\b{\beta}
\def\Lam{\Lambda}

\def\Gam{\Gamma}

\def\a{\alpha}
\def\om{\omega}

\def\Del{\Delta}

\def\vp{\varphi}

\def\gam{\gamma}

\def\Gam{\Gamma}
\def\Lam{\Lambda}
\def\sig{\sigma}
\def\lam{\lambda}
\def\z{\zeta}

\def\e{\varepsilon}
\def\t{\tau}

\def\snm1{\bbs^{n-1}}

\def\intl{\int\limits}

\def\ch{\mathrm{cosh}}

\def\Cs{\mathscr{C}}% c 1234}} Changed

\def\F{{\mathcal F}}

\def\Cs{\mathscr{C c 1234}}

\def\cd{\stackrel{*}{\C}\!{}_{m, k}^\lam}
\def\sd{\stackrel{*}{\S}\!{}_{m, k}^\lam}
\def\cd0{\stackrel{*}{\C}\!{}_{m, k}^\lam}
\def\sd0{\stackrel{*}{\S}\!{}_{m, k}^\lam}

\def\ncd0{\stackrel{*}{\Cs}\!{}_{m, k}^\lam}

\newtheorem{theorem}{Theorem}[section]
\newtheorem{lemma}[theorem]{Lemma}

\theoremstyle{definition}
\newtheorem{definition}[theorem]{Definition}
\newtheorem{example}[theorem]{Example}

\theoremstyle{remark}
\newtheorem{remark}[theorem]{Remark}

\numberwithin{equation}{section}

\theoremstyle{corollary}
\newtheorem{corollary}[theorem]{Corollary}
\newtheorem{proposition}[theorem]{Proposition}

\numberwithin{equation}{section}

\newcommand{\be}{\begin{equation}}
\newcommand{\ee}{\end{equation}}

\newcommand{\bea}{\begin{eqnarray}}
\newcommand{\eea}{\end{eqnarray}}
\newcommand{\Bea}{\begin{eqnarray*}}
\newcommand{\Eea}{\end{eqnarray*}}

\def\sideremark#1{\ifvmode\leavevmode\fi\vadjust{\vbox to0pt{\vss% the remark
 \hbox to 0pt{\hskip\hsize\hskip1em%                          will appear only
\vbox{\hsize2cm\tiny\raggedright\pretolerance10000%          on the side
 \noindent #1\hfill}\hss}\vbox to8pt{\vfil}\vss}}}%
\begin{document}

%\title[Sonar Transforms]
%{From Sonar Transforms to Radon Transforms over Paraboloids and Related Fractional Integrals}

%\title[On the Hemispherical  Transform]
%{On the Hemispherical Radon Transform in the Half-Space}

\title[ Fractional Integrals]
{Fractional Integrals Associated with  Radon Transforms }

%\title[$L^p$-$L^q$ Estimates]
%{$L^p$-$L^q$ Estimates of Fractional Integrals Associated with the Radon Transform over Paraboloids
%}

\author{ B. Rubin}

\address{Department of Mathematics, Louisiana State University, Baton Rouge,
Louisiana 70803, USA}
\email{borisr@lsu.edu}


%\dedicatory{To the memory of Professor Nikolai Karapetovich Karapetyants
% on the occasion of his 80th birthday}

%\date{June 12, 2020 }


\begin{abstract}
We obtain sharp $L^p$-$L^q$ estimates for  fractional integrals generated by Radon transforms of the following three types:
the classical Radon transform over the set of all hyperplanes in $\rn$,
 the Strichartz transversal  transform over only those hyperplanes, which meet the last coordinate axis, and the Radon transform  associated with  paraboloids.
 The method relies on a version
of Stein’s interpolation theorem for analytic families of
operators communicated by L. Grafakos.
  \end{abstract}

\maketitle

%%%%%%%%%%%%%%%%%%

\vspace*{-\baselineskip}

\tableofcontents

\section{Introduction}

The present article grew up from our study of  the Radon-type transforms
\bea
 \label {RRT}  (Rf)(\theta,t)&=&\intl_{\theta^\perp}f(y+t\theta)\,d_\theta y,  \\
\label {bart}
(Tf)(x)&=&\intl_{\bbr^{n-1}} f(y', x_n + x'\cdot y')\, dy',\\
\label {ppar}
(P f)(x)&=&\intl_{\bbr^{n-1}} f(x'-y', x_n -|y'|^2)\, dy',
\eea
arising in  Analysis and applications.  Here
\[
 (\theta,t) \!\in  S^{n-1} \times \bbr,\quad x\!=\!(x_1, \ldots, x_{n-1}, x_n)\!=\!(x', x_n) \!\in \!\rn \; \text {\rm (similarly for $y$)},\]
  $S^{n-1}$ is the unit sphere in $\rn$, $\theta^\perp$  is the hyperplane
orthogonal to $\theta$ and passing through the origin,
$d_\theta y$ denotes the Euclidean measure on $\theta^{\perp}$,  $f$ is a sufficiently good function on $\rn$.
 We also write
\be\label {bwdrt}
(R_{\theta}f)(t)=(Rf)(\theta,t),\qquad (T_{x'}f)(x_n)=(T f)(x', x_n), \ee
\be\label {bwdrt1} (P_{x'}f)(x_n)=(P f)(x', x_n).\ee
 The first
 operator is the well-known hyperplane Radon transform; see, e.g., \cite{H11, Ru15}, and references therein.   The second operator integrates functions on $\rn$ over only those hyperplanes which meet the last coordinate axis. Following Strichartz \cite {Str}, we call it   {\it  the   transversal Radon transform}; see also \cite{Chr, Ru12}, \cite [Section 4.13] {Ru15}.
  The third
 operator was  considered by Christ \cite{Chr} and the author \cite{Ru22}. It  performs integration over shifted paraboloids.  We call it  {\it the parabolic Radon transform}. Different parametrizations of Radon transforms were discussed by Ehrenpreis \cite{E}.
Sharp $L^p$-$L^q$ estimates for $R$  were obtained by Oberlin and Stein \cite {OS}. They yield similar estimates for $T$ and $P$ because  these three operators are intimately connected \cite{Chr, Ru12, Ru15, Ru22}.

The purpose  of the present paper  is two-fold. We plan to examine  known $L^p$-$L^q$ boundedness results for the operators $R$, $T$, and $P$, which were previously obtained with the aid of Stein's interpolation theorem for analytic families of operators \cite {Ste56}. Our second aim is  to study the corresponding analytic families themselves, because they  are of interest on their own right.

 In fact, there exist many such families indexed by a complex parameter $\a$ and associated with the afore-mentioned Radon transforms.  They can be defined using the tools of  single-variable Fractional Calculus and represented by convolutions of the form

\be\label {desds1}
R_{\pm}^\a f \!=\! h^\a_{\pm} \ast R_{\theta}f, \quad   R_{0}^\a f \!=\! h^\a_{0} \ast R_{\theta}f,\qquad  R_{s}^\a f \!=\! h^\a_{s} \ast R_{\theta}f;\quad\ee

\be\label {desds2}
T_{\pm}^\a f \!=\! h^\a_{\pm} \ast T_{x'}f, \quad   T_{0}^\a f \!=\! h^\a_{0} \ast T_{x'}f,\qquad  T_{s}^\a f \!=\! h^\a_{s} \ast T_{x'}f;\ee

\be\label {desds3}
P_{\pm}^\a f \!=\! h^\a_{\pm} \ast P_{x'}f, \quad   P_{0}^\a f \!=\! h^\a_{0} \ast P_{x'}f,\quad  P_{s}^\a f \!=\! h^\a_{s} \ast P_{x'}f.\!\ee
${}$

\noindent Here $h_{\pm}^\a$, $h_{0}^\a$, and $h_{s}^\a$
are single-variable tempered distributions
acting on $R$, $T$, and $P$ in the last variable and defined by  the formulas
\be\label {t2.3}
h^\a_{\pm} (\tau)  =  {\tau^{\a-1}_{\pm}\over \Gamma(\a)}=
\frac{1}{\Gamma(\a)} \left\{
\begin{array}{ll}  |\tau|^{\a-1} &\mbox{\rm if} \; \; \pm \tau >0,\\
0 &\mbox{\rm otherwise;}\\
  \end{array}
\right. \ee
\bea
\label {t2.4} h_0^\a(\tau) &=& {|\tau|^{\a-1}\over\gamma (\a)},\qquad  \a\neq 1, 3, 5, \ldots;\\
\label {t2.4s} h_s^\a(\tau) &=& {|\tau|^{\a-1}\sgn (\t)\over \gamma'(\a)},\qquad \a \neq 2,4,6, \dots;\eea
\[ \gamma (\a)= 2\Gamma(\a) \,\cos (\a\pi/2), \qquad \gamma'(\a)=2i\Gam (\a) \sin (\a \pi/2);\]
see, e.g.,  \cite [Chapter I, Section 3]{GS1},  \cite {Es, Ru13, SKM}.

 All these operators can be treated using the same ideas as
$R_{+}^\a$, $T_{+}^\a$, and $P_{+}^\a$. For the sake of simplicity, we will be focusing only on these three operators and leave  others to the interested reader.
 The corresponding one-dimensional convolution
\be\label {t2.8}
(I^\a_+ \omega)(t) = (h^\a_+ *\omega)(t)= {1\over \Gamma(\a)}\intl^{\infty}_0 \tau^{\a-1} \om (t-\t)\,d\t, \ee
is  the well-known  Riemann-Liouville  fractional integral \cite{SKM}. Fractional integrals $R_{0}^\a f$ were  introduced by V.I. Semyanistyi \cite{Sem}. The operators $R_{0}^\a $, $R_{\pm}^\a$,  and  $R_{s}^\a$ were
 also considered in \cite{Ru13} and  \cite[Section 4.9.3]{Ru15} in the framework of the Semyanistyi-Lizorkin space of Schwartz functions orthogonal to all polynomials (see Definition \ref{Semyanistyi}).

\vskip 0.2 truecm

\noindent {\bf Main results.}  One of the  main results  can be stated as follows.

\begin {theorem} \label {lanseTM} Let $1\le p,q \le \infty$, $\a_0=Re\, \a$.  The operators  $T_{+}^\a$ and $P_{+}^\a$, initially defined on Schwartz functions $f \in S (\rn)$ by analytic continuation, extend as linear bounded operators  from $L^p (\rn)\!$ to $L^q (\rn)\!$ if and only if
\be\label {wz2as}
\frac{1-n}{2} \le \a_0 \le 1, \qquad p=\frac{n+1}{n + \a_0}, \qquad q=\frac{n+1}{1-\a_0}.\ee
\end{theorem}
This statement is a combination of Theorems \ref {lanseT} and \ref{lanseTp}.
The case $\a=0$ gives known results for the  Radon transforms  $T$ and $P$  \cite{Chr, Ru22}.

We also obtain a similar ``if and only if'' statement  for the operator $R_{+}^\a$ in the cases $\a=0$  and $0<Re\, \a <1$;  see Theorems \ref{OST2} and \ref{OST2a}.
Theorem \ref{OST2} represents an alternative version of the Oberlin-Stein boundedness result for the Radon transform $R$.
Theorem \ref{OST2a}, which includes the weak-type estimate  for $p=1$,
 is an analogue of  the Hardy-Littlewood-Sobolev theorem for Riesz potentials; cf. \cite [Chapter V, Section 1.2]{Ste}, \cite [Theorem 0.3.2] {Sog}, \cite [p. 189] {Sog1}).

For  $(1-n)/2 \le Re\, \a \le 0$, $\a \neq 0$, we show that the ``if'' part of Theorem \ref {lanseTM} still holds for $R_{+}^\a$. The validity of the  ``only if'' part for these values of $\a$ is an open problem. We conjecture that  this problem has an affirmative answer, but to prove it, one needs a suitable modification  of Stein's interpolation theorem (see comments below).

 Theorem \ref {lanseTM} for  $T_{+}^\a$ and $P_{+}^\a$, and the corresponding statement for $R_{+}^\a$  guarantee the existence of the $L^p$-$L^q$ bounded extensions of the operators provided by  interpolation. It is natural to ask:

 {\it What are the  explicit analytic formulas for these extensions?}

 Of course, the corresponding expressions can be described in terms of distributions.  However, they belong to $L^q$, and we wonder how these $L^q$-functions look like pointwise. A similar question for solutions of the wave equation was studied in \cite{Ru89}. The case $Re\, \a <0$ is especially intriguing because our operators are not represented by absolutely convergent integrals and need a suitable $L^q$-regularization.
  We shall prove (see Theorem \ref{iffe})
  that the latter can be performed by making use of the  hypersingular integrals, generalizing the concept of Marchaud's fractional derivative \cite {Marc, Ru89a, SKM}.

The plan of the paper is reflected in the Contents.

\vskip 0.2 truecm

\noindent {\bf Comments and challenges.}

{\bf 1.} One of the  motivations for writing this article was the following observation.
A remarkable $L^p$-$L^q$ boundedness result for the Radon transform (\ref{RRT}) was obtained by Oberlin and Stein \cite{OS} by making use of  the operator family  $\{R_{0}^\a\}$  (up to a constant multiple) and
Stein's interpolation theorem  \cite [Theorem 1]{Ste56}.  Although formal application of this theorem yields the correct result, justification of the applicability of the theorem does not seem to be trivial. The crux of the matter is that the assumptions of Stein's theorem and its proof are given in terms of simple functions (finite  linear combinations of the characteristic functions of disjoint compact sets), and it is not so clear how to check the validity of these assumptions for the operator family  $\{R_{0}^\a\}$, which is defined in terms   functions in $S(\rn)$, rather than simple ones.

 Since the  details related to implementation of simple functions were omitted in \cite{OS} and the results of \cite{OS} have already been used in other publications, it was natural to try to find an alternative approach and present it  in full detail.
 To circumvent this obstacle, we invoke fractional integrals associated with the transversal transform $T$. This way
   became possible thanks to the  recent modification of Stein's interpolation theorem due to Grafakos \cite{Graf},  who replaced simple functions by  arbitrary smooth compactly supported functions. That was done  in the case when operators  under consideration take functions on $\rn$ to functions on $\rn$.

   To treat the problem directly and thus cover the ``only if'' part for all $(1-n)/2 \le Re\, \a \le 0$, we need an interpolation theorem like that in \cite{Graf}, but for the case when the source space and the target space are not necessarily the same. In our case, the
   target space must be $S^{n-1} \times \bbr$. To the best of my knowledge, such a more general interpolation theorem is not available in the literature and is highly desirable for the study of diverse analytic families arising in Analysis and its applications.

 Regarding the general context of metric measure spaces, including the case when the source space and the target space may be different, one should  mention a recent paper  by Grafakos and Ouhabaz \cite{GO}. In this paper,  simple functions are substituted  by  continuous  compactly supported functions, not necessarily smooth. The interpolation theorem  from \cite{GO}  is not directly applicable to the operator families (\ref{desds1})-(\ref{desds3}), because  the smoothness of $f$ is crucial for the their definition when $\a$ is negative.

Note also that  unlike \cite{OS},  to minimize technicalities  related to  fractional powers of the Laplacian, we prefer to work with the analytic family  $\{R_{+}^\a\}$, rather than $\{R_0^\a\}$.

\vskip 0.2 truecm

{\bf 2.} Some $L^p$-$L^q$ estimates for the localized modifications of $P_{\pm}^\a f $  and $Pf$ with a
 smooth  cut-off function under the sign of integration were announced   by
 Littman \cite {Litt} and Tao \cite{Tao}, who also referred to   Stein's interpolation theorem in  \cite {Ste56}.
 In contrast,  our operators are not localized and we use another interpolation technique, which is based on \cite{Graf}.

\vskip 0.2 truecm

\noindent {\bf Acknowledgements.} The author is grateful to Professor Loukas Grafakos for useful discussions and sharing his knowledge of the subject. Special thanks go to Professor Daniel M. Oberlin for  correspondence.

\section{Preliminaries}

\subsection{Notation} ${}$\hfill

In the following, $x\!=\!(x_1, \ldots, x_{n-1}, x_n)\!=\!(x', x_n) \!\in \!\rn$; $S^{n-1}$ is the unit sphere in $\rn$ with the surface area measure $d\theta$;
$\sigma_{n-1} \equiv \int_{\sn} d\theta= 2\pi^{n/2} \big/ \Gamma (n/2)$ is the area of $S^{n-1}$.
The notation $C(\bbr^n)$,  $C^\infty (\bbr^n)$,
and $L^p (\bbr^n)$ for function spaces is standard;   $||\cdot ||_p =||\cdot ||_{L^p (\rn)}$; $C_0 (\bbr^n)=\{f\in C(\bbr^n):
\lim\limits_{|x|\to\infty} f(x) = 0\}$; $C_c^\infty (\bbr^n)$ is the space of compactly supported infinitely differentiable functions on $\rn$.
 A similar notation will be used for functions on the cylinder $Z_n=S^{n-1}\times \bbr$. The corresponding $L^p$-norms will be denoted by  $||\cdot ||^{\sim}_p$.

 The notation $\langle f,g\rangle$ for  functions $f$ and $g$ is used for the integral of the product of these functions.
 We keep the same notation  when $f$ is a distribution and $g$ is a test function.

 If $m=(m_1, \ldots,  m_n)$ is a multi-index, then $\partial^m = \partial_1^{m_1} \ldots     \partial_n^{m_n}$,  where  $\partial_i  =\partial/ \partial x_i$;  $\Delta = \partial^2/\partial x_1^2 +\ldots +\partial^2/\partial x_n^2$ is
 the Laplace operator.
The Fourier transform of a function
$f \in  L^1 (\bbr^n)$ is defined by
\be \label{ft} \hat f (\xi) = \intl_{\bbr^{ n}} f(x) \,e^{ i x \cdot \xi} \,dx, \qquad \xi\in \rn,
\ee where $x \cdot \xi = x_1 \xi_1 + \ldots + x_n\xi_n$.
 We denote by $S(\bbr^n)$
the Schwartz space of $C^\infty$-functions
which are rapidly
decreasing together with their derivatives of all orders. The space $S(\bbr^n) $ is equipped  with the topology generated by the sequence
of norms
\be\label{setop} ||\vp||_k=\max\limits_x(1+|x|)^k
\sum_{|m|\le k} |(\part^m\vp)(x)|, \quad k=0,1,2, \ldots .\ee

The Fourier transform is an isomorphism of $S(\bbr^n)$.
The space of tempered distributions, which is dual to $S(\bbr^n)$, is denoted by $S'(\bbr^n)$.  The Fourier transform of a distribution $f\in S'(\bbr^n)$ is  a
distribution $\hat f\in S'(\bbr^n)$ defined by
\be\label{ftrd12}
\langle \hat f,\psi\rangle=\langle f,\hat \psi\rangle,\qquad \psi\in S(\bbr^n). \ee
The equality (\ref{ftrd12}) is equivalent to
\be\label{ftrd12V}
\langle\hat f,\hat\vp\rangle=(2\pi)^n \langle f,\vp_1 \rangle,\qquad \vp\in S(\bbr^n), \quad \vp_1 (x)=\vp (-x). \ee
The inverse Fourier transform of a function (or distribution) $f$ is denoted by $\check f$.

Given a  real-valued quantity
$X$ and a complex number $\lam$, we set $(X)_\pm^\lam= |X|^\lam$ if $\pm X>0$ and $(X)_{\pm}^\lam=0$, otherwise.
 All integrals are understood in the  Lebesgue sense.  The letter $c$, sometimes with subscripts, stands for a nonessential constant that may be different at each
occurrence.  We write $\bbn$ for the set of all positive integers; $\bbz_+$ denotes  the set of all nonnegative integers;
$\bbz_+^n$ is the set of all multi-indices $\gam =(\gam_1, \gam_2, \ldots, \gam_n)$,  $\gam_i \in \bbz_+$ $(i=1,2,\ldots, n)$. The notation $e_1, \ldots, e_n$ will be
used for coordinate unit vectors in $\rn$; $O(n)$ is the group of orthogonal transformation of $\rn$ equipped with the probability Haar measure.

\subsection{The Semyanistyi-Lizorkin spaces} \label {orkin}
To circumvent some obstacles arising in the study of convolutions whose Fourier transforms have singularities, it may be convenient to
 reduce the  space $S(\bbr^n)$ of the test functions in a suitable way and expand the corresponding  space of distributions.
An idea of this approach originated in the work of Semyanistyi \cite{Sem} who treated the Fourier multipliers having singularity at a single point $x=0$. This idea was independently used by Helgason \cite [p. 162]{H65}. It was
extended by Lizorkin \cite {Liz} and later by Samko \cite{Sam77} to more general sets of singularities; see also \cite{Ru15, Sam,  SKM, Yos} and references therein.  In the present paper,  the set of singularities is a  either the hyperplane $x_n=0$ or a single point $x=0$.

\begin{definition}\label {Semyanistyi}
{\it For $n\ge 2$, we denote by $\Psi(\bbr^n)$,
   the subspace of  functions $\psi \in S(\bbr^n)$
vanishing with all derivatives on the hyperplane $x_n=0$.
 Let $\Phi(\bbr^n)$  be the
  Fourier image of $ \Psi (\bbr^n)$. For $n\ge 1$, we denote by   $\Psi_0(\bbr^n)$ the subspace of functions in $S(\bbr^n)$
vanishing with all derivatives at the origin $x=0$. The
  Fourier image of $ \Psi_0 (\bbr^n)$ will be denoted by $\Phi_0(\bbr^n)$. }
  \end{definition}

 \begin{proposition}  \label{jikl1} {\rm (cf. \cite[Proposition  4.141]{Ru15}) }
 A  function $\vp \in S(\bbr^n) $ belongs to
$\Phi (\bbr^n)$ if and only if
  \be\label{kssi1} \intl_{\bbr} \vp(x', x_n)\, x_n^{k}
 \,dx_n=0 \quad \text {for all} \quad  k\in \bbz_+ \;\text {and all} \quad x' \in\bbr^{n-1}.\ee
 The space $\Phi_0(\bbr^n)$ consists of Schwartz functions orthogonal to all polynomials.
 \end{proposition}

\begin{proposition} \label{izsh} {\rm (cf.  \cite[Theorem 2.33]{Sam})}  The spaces $\Phi (\bbr^n)$ and $\Phi_0 (\bbr^n)$ are  dense in $L^p (\rn)$, $1<p<\infty$.
 \end{proposition}

\begin{proposition} \label{iqah}  {\rm (cf.  \cite[Lemma 3.11]{Ru96}, \cite{Yos})} Let $1\le p,q < \infty$.  If $f\in L^p (\rn)$ and  $g \in L^q (\rn)$ coincide
as $\Phi'$-distributions or as $\Phi_0'$-distributions, then they coincide almost everywhere on $\rn$.
 \end{proposition}

\section {Two modifications of the Radon transform over hyperplanes}

 Let $\Pi_n$ be  the set of all unoriented hyperplanes in
$\bbr^n$.  The Radon transform of a sufficiently good  function $f$ on $\rn$
  is a function on $\Pi_n$  defined by the formula
\be\label{rtra} (Rf)(h)= \intl_{h} f(x) \,d_h x,\qquad h \in
\Pi_n,\ee where $d_h x$ stands for the induced Euclidean measure on $h$. To study this transform using the tools of Analysis, we need to endow
 $\Pi_n$ with parametrization which converts $Rf$ into a function on some measure space. Several different parametrizations of $\Pi_n$  are known; see, e.g., \cite[Section 4.1]{Ru15}.
 The most common one (cf. \cite {H11}) is the  parametrization of
$h \in \Pi_n$  by the pair $(\theta,t) \in  S^{n-1}\times \bbr$, so that
\be\label{hppl} h\equiv h
(\theta,t)=\{x\in \bbr^n: x \cdot \theta=t\}. \ee
 One can think of $(\theta,t)$  as a
point of the cylinder $ Z_n=S^{n-1}\times \bbr$ equipped with the product measure $d\theta dt$. Since
\be\label{oo90der} h
(\theta,t)=h (-\theta,-t),\ee the equality (\ref{hppl}) realizes a
two-to-one correspondence between $Z_n$ and
$\Pi_n$. Thus, every function $\vp$ on $\Pi_n$ can be identified with an even function $\vp (\theta,t)$ on  $Z_n$, so that
 \be\label{adas}\vp (\theta,t)=\vp (-\theta,-t) \quad \forall\; (\theta,t) \in Z_n.\ee
 In the $(\theta,t)$-language, the Radon transform (\ref{rtra}) has the form
\be\label{rtra1} (Rf) (\theta, t)\equiv (R_\theta f) (t)=\intl_{\theta^{\perp}} f(t\theta +
y) \,d_\theta y,\ee where $\theta^\perp=\{x: \, x \cdot \theta=0\}$ is the hyperplane
orthogonal to $\theta$ and passing through the origin, and
$d_\theta y$ denotes the Euclidean measure on $\theta^{\perp}$. Clearly, \be \label{972sert} (Rf) (\theta,t) = (Rf)(-\theta, -t).\ee
 If $f$ is a  radial function,   $f(x) \!\equiv \!f_0(|x|)$, then
\be\label{rese}
 (Rf)(\theta, t) = \sigma_{n-2} \intl^\infty_{|t|}\! f_0 (r)
(r^2-t^2)^{(n-3)/2}r dr\ee
for all $\theta \in S^{n-1}$; see, e.g., \cite [Lemma 4.17]{Ru15}.

A certain inconvenience of the $(\theta,t)$-parametrization is that  the source space $\rn$ and the target space  $Z_n$ have different geometry and the harmonic analysis on $Z_n$ is more involved than that on $\rn$.

An alternative parametrization can be used if we eliminate a subset of  all hyperplanes which are parallel to the $x_n$-axis. This subset has measure zero and can be ignored in some considerations.
 The remaining set $\tilde\Pi_n$ contains only those  hyperplanes which  meet the last coordinate axis.
 Every hyperplane  $h \in \tilde\Pi_n$  can be parametrized by a point $x=(x',x_n)\in \rn$, so that
 \be
\label{par1} h=\{y=(y',y_n)\in \rn: \, y' \in \bbr^{n-1}, \; y_n=x'\cdot y' + x_n\}.\ee
The corresponding Radon-type transform has the form
\be\label {bart22}
(Tf)(x)\equiv (T_{x'}f)(x_n) =\intl_{\bbr^{n-1}} f(y', x_n + x'\cdot y')\, dy',\quad x\in \rn. \ee
The dual transform $T^*$, satisfying $\langle Tf,g \rangle=\langle f, T^* g \rangle$, has a similar form
\be\label {dbart}
(T^*g)(x)\equiv (T^*_{x'}g)(x_n)  =\intl_{\bbr^{n-1}} g(y', x_n - x'\cdot y')\, dy'. \ee
Clearly,
\be\label {dbart7}
  (T^*g)(x', x_n)\!=\! (Tg)(-x', x_n)\!=\!(T\dot g)(x', -x_n), \ee
where $\dot g (y', y_n)\!=\!g (y', -y_n)$.

An advantage of $T$ in comparison with $R$ is that $T$ maps functions on $\rn$ to functions on $\rn$.
 By the standard  Calculus,
 \be\label{con1} (Rf)(\theta, t)=(Rf)(h)=\sqrt{1+|x'|^2}\,
(T f)(x',x_n).\ee
 Different  parametrizations  are related by
 \be \label{change1}
 \theta=\frac{x'-e_n}{\sqrt{1+|x'|^2}}, \qquad
t=-\frac{x_n}{\sqrt{1+|x'|^2}},\ee
\be \label{change2}
x'=-\frac{\theta'}{\theta_n}, \qquad x_n=\frac{t}{\theta_n}.\ee
The corresponding  transition mappings are defined on functions $\vp (x)\equiv\vp(x',x_n)$ and $\psi (\theta, t)$ by
\be \label{tonv} (\Lam_0\vp)(\theta,
t)=|\theta_n |^{-1}\vp \left (-\frac{\theta'}{\theta_n},\frac{t}{\theta_n}\right),
\ee
\be\label{tinv} (\Lam_0^{-1} \psi)(x)=(1+|x'|^2)^{-1/2} \psi \left
(\frac{x'-e_n}{\sqrt{1+|x'|^2}}, -\frac{x_n}{\sqrt{1+|x'|^2}}\right).\ee

\begin{lemma}
The following equalities  hold:
 \bea \label{45}(Rf)(\theta, t)&=& (\Lam_0 T f)(\theta, t), \qquad  \theta_n \neq 0;\\
\label{46v} (T f)(x)&=& (\Lam_0^{-1} Rf)(x).\eea
Moreover, for every $1\le p<\infty$ and a real number $\nu$,
\be\label {tae}
\intl_{Z_n} |t|^{\nu p} |(Rf)(\theta, t)|^p \,  d\theta dt= 2 \intl_{\bbr^{n}} \!\!\frac{|x_n|^{\nu p}\,  |(T f) (x)|^p }{(1\!+\!|x'|^2)^{(n-p+\nu p +1)/2}}\, dx.\ee
It is assumed  that either side of the corresponding equality exists in the Lebesgue sense.
\end{lemma}
\begin{proof}
 The  equalities  (\ref{45}) and (\ref{46v}) follow from (\ref{con1}). To prove (\ref{tae}), we write the left-hand side as
\[I\equiv \intl_{\bbr} |t|^{\nu p} dt \intl_{S^{n-1}} |\theta_n |^{-p} \left |  (T f)\left (-\frac{\theta'}{\theta_n},\frac{t}{\theta_n}\right)\right |^p d\theta.\]
The integral over $S^{n-1}$ can be converted into the integral over $\bbr^{n-1}$ by the formula \cite [Proposition 1.37 (i)]{Ru15}
 \be\label{teq1} \intl_{S^{n-1}}\!\!\!
F (\theta)\, d\theta\!=\!\!\intl_{\bbr^{n-1}}\!\!\left [F \Big (\frac{x'\!+\!e_n}{|x'\!+\!e_n|} \Big )\!+\!F \Big
(\frac{x'\!-\!e_n}{|x'\!-\!e_n|} \Big )\right ] \frac{dx'}{(1\!+\!|x'|^2)^{n/2}}\,.\ee
In our case,
\[
F (\theta)= |\theta_n |^{-p} \left |  (T f)\left (-\frac{\theta'}{\theta_n},\frac{t}{\theta_n}\right)\right |^p.\]
Hence
\bea
I\!\!&=&\!\!
 \intl_{\bbr} |t|^{\nu p} dt \nonumber\\
&\times&\!\! \intl_{\bbr^{n-1}} \!\!\left [ |(T f) (-x'\! +\!t\sqrt{1\!+\!|x'|^2} e_n)|^p\! + \! |(T f) (x' \!-t\sqrt{1\!+\!|x'|^2} e_n)|^p\right ]\nonumber\\
&\times& \frac{dx'}{(1+|x'|^2)^{(n-p)/2}}.\nonumber\eea
Changing variables, we obtain
\[
I=2\intl_{\bbr^{n}} \!\!\frac{|x_n|^{\nu p}}{(1\!+\!|x'|^2)^{(n-p+\nu p +1)/2}} \,|(T f) (x)|^p \,dx, \]
as desired.
\end{proof}

Both transforms $R$ and $T$ are well defined  on functions $f\in L^1(\rn)$ because by Fubini's theorem,
\be\label{vali}\intl_{-\infty}^\infty  \!(Rf) (\theta, t)dt\! = \!\intl_{-\infty}^\infty (Tf) (x', x_n) dx_n\!=\!\intl_{\bbr^n}\! f(x) dx. \ee

The  Radon transforms $R$, $T$,  and  $\tilde T$ are intimately related to  the
  Fourier transform.
\begin{theorem}\label{pror4} Let $f \in  L^1(\bbr^n)$.
Then
\be\label{prsl}  (R_\theta f)^\wedge (\rho)=\hat f(\theta\rho), \qquad \rho \in \bbr, \quad \theta \in S^{n-1}; \ee
\be\label{prslT}
(T_{x'}f)^\wedge(\xi_n)=\hat f(-x'\xi_n, \xi_n), \qquad   (T^*_{x'}f)^\wedge(\xi_n)=\hat f(x'\xi_n, \xi_n).  \ee
\end{theorem}
 \begin{proof} The statement (\ref{prsl}) is well known in the literature as  the  {\it Fourier slice theorem}; see, e.g., \cite[p. 11]{Nat},  \cite [p. 129]{Ru15}.
 The equalities (\ref{prslT}) have the same meaning for the transversal transform and its dual.
 The first equality in (\ref{prslT}) can be found in \cite[Lemma 4.143]{Ru15} and
 its proof is straightforward:
 \bea
(T_{x'}f)^\wedge(\xi_n)&=&\intl_{-\infty}^\infty e^{i x_n \xi_n} dx_n \intl_{\bbr^{n-1}}  f(y', x_n + x'\cdot y')\, dy'\nonumber\\
&=&\intl_{\rn} f(y)  e^{i  \xi_n (y_n-  x'\cdot y')} dy=\hat f(-x'\xi_n, \xi_n).\nonumber\eea
The second equality in (\ref{prslT}) follows from the first one if we note that  $(T^*f)(x', x_n) =(Tf)(-x', x_n)$; cf. (\ref{dbart}).
\end{proof}

\begin{lemma}\label{lscerd} \cite[formulas (4.4.4), (4.13.6)]{Ru15} The following relations hold:
  \bea\label{duas3} \intl_{Z_n} \frac{(Rf)(\theta,
t)}{(1+t^2)^{n/2}}\,d\theta dt&=& \sig_{n-1}\intl_{\bbr^n}
\frac{f(x)}{(1+|x|^2)^{1/2}}\,dx, \\
 \label{eq2}\intl_{\bbr^n} \frac{ (Tf)(x)}{(1+|x|^2)^{n/2}} \,dx
&=&\frac{\sig_{n-1}}{2}\, \intl_{\rn}\frac{ f(x)}{(1+|x|^2)^{1/2}} \,dx.  \eea
It is assumed that either side of the corresponding equality exists in the Lebesgue sense.
\end{lemma}

\begin{theorem} \label{lsu} If $f\in L^p (\rn)$, $1 \le p < n/(n-1)$, then $(Rf)(\theta, t)$ and $(Tf)(x)$ are finite for almost all $(\theta, t) \in Z_n$ and $x\in \rn$, respectively. The
left-hand sides of (\ref{duas3}) and (\ref{eq2}) do not exceed $ c\, ||f||_p$, $\, c=\const$.
\end{theorem}

This statement follows from (\ref{duas3}) and (\ref{eq2})  by H\"older's inequality.

\begin{remark}\label {kaz} The bound $p < n/(n-1)$ in Theorem \ref{lsu} is sharp. For example, if $p \ge n/(n-1)$ and
 \be\label{emily1} f_0(x)=(2+|x|)^{-n/p} (\log (2+|x|))^{-1} \quad (\in
L^p(\bbr^n)),\ee
 then $(Rf_0) (\theta, t)\equiv \infty$; see, e.g., \cite[Theorem 4.28]{Ru15}. By (\ref{46v}), it follows that
$ T f_0\equiv \infty$.
\end{remark}

The next statements provide additional  information about the  mapping properties of $R$ and $T$.

\begin{theorem}\label {o9iur} \cite [Theorem 4.32] {Ru15} Let
\be\label{ovdN} 1\le p < n/(n-1), \qquad \nu=-(n-1)/p', \qquad 1/p+1/p'=1.\ee
 We introduce  the weighted  space
 \be\label {kakx}
L^p_\nu (Z_n)=\{ \vp: ||\vp||^\sim_{p,\nu}\equiv |||t|^\nu \vp||_{L^p(Z_n)}<\infty \}.\ee
Then $||R f||^\sim_{p,\nu}\le A \,||f||_{p}$, where
\be\label{op983r}
A=||R||=\frac{\displaystyle  {2^{1/p}\, \pi^{(n-1)/2}\,\Gam \left(\frac{1 -n/p'}{2}\right)}}{\displaystyle{\Gam \left(\frac {n}{2p}\right)}}\,
.\ee
\end{theorem}

 The  necessity of the condition $\nu=-(n-1)/p'$  can be proved using the standard scaling argument $x \to \lam x$, $\lam >0$; see, e.g., the proof of Theorem \ref {o9iuN} in the more general situation.

Theorem \ref {o9iur} yields  the corresponding result for  $T$.

 \begin{theorem}\label {o9iu2t}  Let $1\le p < n/(n-1)$,
\[
u(x) = \frac {|x_n|^\nu}{(1+|x'|^2)^{\mu/2}}, \qquad \mu = -n (1-2/p), \qquad \nu=-(n-1)/p',\]
\be\label {space} L^p_u(\bbr^n)=\{ f: ||f||_{p,u}\equiv || u(x) f||_p <\infty \}.\ee
Then
\be\label{op982yf}
||T f||_{p,u}\le  2^{-1/p} A \,||f||_{p},\ee
where the constant $A$ is defined by  (\ref{op983r}).
\end{theorem}
 \begin{proof}  Set $\nu=-(n-1)/p'$ in (\ref{tae}).
  \end{proof}

Let us also consider the Radon transforms of functions $f\in S(\rn)$.

\begin{lemma}\label {from}
If $f\in S(\rn)$, then $T f \in C^\infty (\rn)$. Moreover, the function $x_n \to (Tf)(x', x_n)$ belongs to $S(\bbr)$ for each $x'\in \bbr^{n-1}$, and the function $x' \to (Tf)(x', x_n)$  is tempered for each $x_n \in \bbr$.
\end{lemma}
\begin{proof} The first statement is obvious, because we can differentiate under the sign of integration in (\ref{bart22}) infinitely many times. Furthermore,  for any positive integers $\ell$ and $m$, there is a constant $c_{\ell, m}$ such that
\[
|f(y', x_n + x'\cdot y')| \le \frac {c_{\ell, m}}{(1+|y'|)^\ell  \,  (1+|x_n + x'\cdot y'|)^m }.\]
Note that
\bea
&&\frac {1}{1+|x_n + x'\cdot y'|}=\frac {1+|x_n + x'\cdot y' - x'\cdot y'|}{(1+|x_n + x'\cdot y'|)(1+|x_n|)}\nonumber\\
&&{}\nonumber\\
&&\le\frac{1}{1+|x_n|}\Big(1+\frac{|x'\cdot y'|}{1+|x_n + x'\cdot y'|}\Big ) \le \frac{1+|x'\cdot y'|}{1+|x_n|}\nonumber\\
&&{}\nonumber\\
&&\le \frac{1+2|x'\cdot y'|}{1+|x_n|}\le \frac{1+|x'|^2 +|y'|^2}{1+|x_n|}\nonumber\\
&&{}\nonumber\\
&&\le  \frac{(1+|x'|^2)(1 +|y'|^2)}{1+|x_n|}.\nonumber\eea
Hence
\bea\label
|(Tf)(x', x_n)| &\le& c_{\ell, m} \,\frac{(1+|x'|^2)^m}{(1+|x_n|)^m} \intl_{\bbr^{n-1}} \frac{(1+|y'|^2)^m}{(1+|y'|)^\ell} \, dy'\nonumber\\
\label{pari} &\le&  \tilde c_{\ell, m} \,\frac{(1+|x'|^2)^m}{(1+|x_n|)^m}, \qquad \tilde c_{\ell, m}=\const,\eea
 if $\ell$ is  big enough in comparison with $m$. A similar estimate holds for any derivative of $Tf$. This gives other statements of the lemma.
\end{proof}

An analogue of Lemma \ref {from}  for the Radon transform $(Rf) (\theta, t)$ is more involved and needs additional definitions.
\begin{definition} \label{iiuuyg5a}
A function $g$ on $S^{n-1}$ is called differentiable
if the corresponding homogeneous function \[\tilde g(x)= g(x/|x|)\] is
differentiable in the usual sense on $\bbr^n \setminus \{0\}$. The
derivatives of a function $g$ on $S^{n-1}$ will be defined as
restrictions to $S^{n-1}$ of the  derivatives of
$\tilde g(x)$, namely,
\be\label{iiuuyg}
(\partial^\gam g)(\theta)=(\partial^\gam \tilde g)(x)\big|_
{x=\theta},\qquad \gam \in \bbz_+^n, \quad \theta\in S^{n-1}.
\ee
\end{definition}

\begin{definition} \label{iiuuyg5a1}
 We denote by $ S(Z_n)$ the  space of
functions $\vp(\theta,t)$ on $Z_n=S^{n-1}\times \bbr$, which are
infinitely differentiable in $\theta$ and $t$ and rapidly decreasing
as $t \to \pm\infty$ together with all derivatives.
The topology in $ S(Z_n)$ is defined by the sequence of norms
\be\label {jjyyz}
||\vp ||_m=\sup\limits_{|\gam|+j\le m} \sup\limits_{\theta,t} \
\!(1+|t|)^m |(\partial_\theta^\gam \partial^j_t \vp)(\theta,t)|, \quad
m\in \bbz_+.
\ee
The subspace of even functions in $ S(Z_n)$ will be denoted by $ S_e(Z_n)$.
\end{definition}

\begin{theorem} \label {mnub} \cite[Theorem 4.47]{Ru15}, \cite[Chapter I, Theorem 2.4]{H11} The Radon transform $f \to Rf$ is a linear continuous map from $ S(\bbr^n)$  into  $S_e(Z_n)$.
\end{theorem}

%{\stackrel{*}{T}\!{}_{m, k}^\lam}

\section {Radon-type fractional integrals $R_{+}^\a f$  and $T _{+}^\a f$}

In this section we
  consider   fractional integrals of the form
\bea\label{tongcNN} (R_{+}^\a f)(\theta, t))&=&\frac{1}{\Gam(\a)}\intl_{\bbr^n}\!\!
(t\!-\!x\cdot\theta)_{+}^{\a -1} f(x)\, dx,  \\
\label {99jh1}
(T_{+}^\a f)(x)&=& \frac{1}{\Gam (\a)}\intl_{\rn} (x_n -y_n)_{+}^{\a -1} f(y', y_n + x' \cdot y')\, dy,\\
\label {d99jh1}
(\stackrel{*}{T}\!{}_{+}^\a g)(x)&=& \frac{1}{\Gam (\a)}\intl_{\rn} (x_n -y_n)_{-}^{\a -1} g(y', y_n - x' \cdot y')\, dy.\eea
If $Re \,\a>0$ and  $f$ is good enough, these integrals are absolutely convergent. The limiting case $\a=0$ in (\ref{tongcNN}) and (\ref{99jh1}) yields $Rf$ and $Tf$, respectively. The limiting case $\a=0$ in  (\ref{d99jh1}) gives $T^*f$, as in (\ref{dbart}).
One can readily see that
\[\langle T_{+}^\a f, g \rangle = \langle f, \stackrel{*}{T}\!{}_{+}^\a g \rangle\]
 and
\be\label{tocd}
 (\stackrel{*}{T}\!{}_{+}^\a g)(x', x_n)= (T_{+}^\a \dot g)(x', -x_n), \qquad \dot g (y', y_n)= g (y', -y_n).\ee

Given two functions $\vp (x)\equiv\vp(x',x_n)$ and $\psi (\theta, t)$, denote
\be\label {lab1} (\Lam_\a\vp)(\theta,
t)=|\theta_n |^{\a-1}\vp \left (-\frac{\theta'}{\theta_n},\frac{t}{\theta_n}\right),
\ee
\be\label{lab2}  (\Lam_\a^{-1} \psi)(x)=(1+|x'|^2)^{(\a-1)/2} \psi \left
(\frac{x'-e_n}{\sqrt{1+|x'|^2}}, -\frac{x_n}{\sqrt{1+|x'|^2}}\right);\ee
cf. (\ref{tonv}), (\ref{tinv}) for $\a=0$.

\begin{lemma} \label {ordi}
The following relations hold:
\bea \label{4A5i}(R_{+}^\a f)(\theta, t)&=& (\Lam_\a T_{+}^\a f)(\theta, t), \qquad  \theta_n \neq 0;\\
\label{4A6v} (T_{+}^\a f)(x)&=&(\Lam_\a^{-1} R_{+}^\a f)(x).\eea
Moreover, for every $1\le p<\infty$ and a real number $\nu$,
\be\label {tae1}
\intl_{Z_n} |t|^{\nu p} |(R_{+}^\a f)(\theta, t)|^p dt d\theta= 2 \intl_{\bbr^{n}} \!\!\frac{|x_n|^{\nu p}\,  |(T_{+}^\a  f) (x)|^p }{(1\!+\!|x'|^2)^{(n+(\a -1+\nu)p  +1)/2}}\, dx.\ee
It is assumed  that either side of the corresponding equality exists in the Lebesgue sense.
\end{lemma}
 \begin{proof} By (\ref{change1}),
\bea(R_{+}^\a f)(\theta, t)&=&  \frac{1}{\Gam (\a)} \intl_{\rn}  (t- y\cdot \theta )_{+}^{\a -1}  f(y)\, dy\nonumber\\
 &=& \frac{1}{\Gam (\a)} \intl_{\rn}   \left (\frac{x_n + x' \cdot y' -y_n}{\sqrt{1+|x'|^2}}\right )_{\!+}^{\a -1}  f(y)\, dy\nonumber\\
&=&\frac{(1+|x'|^2)^{(1-\a)/2}}{\Gam (\a)} \intl_{\rn} (x_n -y_n )_{+}^{\a -1} f (y', y_n + x' \cdot y')\, dy\nonumber\\
&=&(1+|x'|^2)^{(1-\a)/2} (T_{+}^\a f)(x).\nonumber\eea
This gives (\ref{4A6v}).  The equality (\ref{4A5i}) is a consequence of (\ref{4A6v}). The proof of (\ref{tae1}) is an almost verbatim copy of the  proof of (\ref{tae}).
\end{proof}

\subsection {Fractional integrals of $L^p$-functions}
 Consider the weighted space
  \[
L^p_\nu (Z_n)=\Big \{ \vp: ||\vp||^\sim_{p,\nu}\equiv \Big (\intl_{Z_n} |t|^{\nu p}\,  |\vp (\theta, t)|^p \, d\theta dt \Big )^{1/p}<\infty \Big \},\]
where $Z_n= S^{n-1} \times \bbr$, $1\le p <\infty$, and  $\nu$ is a real number. Equivalently,
\be\label{zzxsy}
||\vp||^\sim_{p,\nu}\!=\! \Bigg ( \sig_{n-1}\intl_{-\infty}^\infty \!\!|t|^{\nu p}\,dt\!\intl_{O(n)} \!|\vp (\gam e_1, t)
|^p\,d\gam\Bigg )^{1/p}.\ee

\begin{theorem}\label {o9iuN}  Let  $\,0<\a <1$,
\be\label{Nop98yoN} 1\le p < \frac{n}{n-1+\a}, \qquad \nu= -\a-\frac {n-1}{p'}, \ee
$1/p +1/p' =1$. Then
\be\label{ortN}
||R_{+}^\a f||^\sim_{p,\nu}\le c_\a \,||f||_{p},\qquad c =c\,(\a,n,p)=\const.\ee
The conditions  (\ref{Nop98yoN}) are sharp. There is a function $\tilde f \in L^p(\bbr^n)$ with $p \ge n/(n-1+\a)$, such that $(R_{+}^\a \tilde f)(x) \equiv \infty$.
\end{theorem}
\begin{proof}  Denote $\vp (\theta, t)=(R_{+}^\a f)(\theta, t)$
  and consider the cases $t>0$ and $t<0$ separately.
  For $t>0$, changing variables $\theta=\gam e_1$, $\gam \in O(n)$, and $x=t\gam y$, we get
\be\label{tongyi} \vp (\gam e_1, t)=\frac{t^{\a+n-1}}{\Gam (\a)}\intl_{\bbr^n}
f(t\gam y)\, (1-y_1)_{+}^{\a -1} \, dy. \ee
For $t<0$, we similarly have
\be\label{tongyiT} \vp (\gam e_1, t)=\frac{|t|^{\a+n-1}}{\Gam (\a)}\intl_{\bbr^n}
f(|t|\gam y)\, (-1-y_1)_{+}^{\a -1} \, dy. \ee
Splitting the integral $\int_{-\infty}^\infty$ in  (\ref{zzxsy}) in two pieces,  we  write
\[||\vp||^\sim_{p,\nu}= \frac{\sig_{n-1}^{1/p}}{\Gam (\a)} \,(I_1 +I_2)^{1/p},\]
 where
\bea
I_1^{1/p}&\le& \intl_{\bbr^n} (1\!-\!y_1)_{+}^{\a -1}
\Bigg (\,\intl_0^\infty \intl_{O(n)} \!t^{(\a+n-1+\nu) p}\,|f (t\gam y)|^p \,d\gam dt \Bigg )^{1/p} \!\!\!dy\nonumber\\
&=&c_1 \, ||f||_{p},\qquad  c_1=\sig_{n-1}^{-1/p} \intl_{\bbr^n} (1\!-\!y_1)_{+}^{\a -1} |y|^{-n/p}\, dy\nonumber \eea
(here we have used the equality $\nu=-\a-(n-1)/p'$). Similarly,
\bea
I_2^{1/p}&\le& \intl_{\bbr^n} (-1\!-\!y_1)_{+}^{\a -1}
\Bigg (\,\intl_{-\infty}^0 \intl_{O(n)} |t|^{(\a+n-1+\nu) p}\,|f (|t|\gam y)|^p \,d\gam dt \Bigg )^{1/p} \!\!\!dy\nonumber\\
&=&c_2 \, ||f||_{p},\qquad  c_2=\sig_{n-1}^{-1/p} \intl_{\bbr^n} (-1\!-\!y_1)_{+}^{\a -1} |y|^{-n/p}\, dy.\nonumber \eea
One can easily check that  $c_1$ and $c_2$ are finite. Hence (\ref{ortN}) follows.

 The necessity of the equality $\nu=-\a-(n-1)/p'$ can be proved using the  scaling argument. Specifically, let $f_\lam (x)=f(\lam x)$, $\lam>0$. Then
\[ ||f_\lam ||_{p}=\lam^{-n/p}\,||f||_{p}, \qquad ||R_+^\a f_\lam ||^\sim_{p,\nu}=\lam^{-\nu-n-\a+1/p'}\,||R_+^\a f||^\sim_{p,\nu}.\]
Hence, if $||R_+^\a f||^\sim_{p,\nu}\le c \,||f||_{p}$ with $c$ independent of $f$, then, necessarily,
$\nu=-\a-(n-1)/p'$.

 Let us show  the necessity of the bound  $p< n/(n-1+\a)$. Consider, for example,
\[ \tilde f(x)=(2+|x|)^{-n/p} (\log (2+|x|))^{-1} \quad (\in
L^p(\bbr^n)).\]
This function is radial. We set $\tilde f(x)\equiv f_0 (|x|)$. Hence, by (\ref{rese}),
\bea
(R_{+}^\a \tilde f)(\theta, t)&=&\frac{1}{\Gam(\a)}\intl_{-\infty}^t (t-s)^{\a -1} (R \tilde f)(\theta,s) ds\nonumber\\
&=&\frac{\sigma_{n-2}}{\Gam(\a)}\intl_{-\infty}^t (t-s)^{\a -1} ds \intl^\infty_{|s|}\! f_0 (r)(r^2-s^2)^{(n-3)/2}r dr.\nonumber\eea
Suppose $t<0$ and change variables $t=-\t$, $s=-\z$. We obtain
\bea
&&(R_{+}^\a \tilde f)(\theta, -\t)\!=\!\frac{\sigma_{n-2}}{\Gam(\a)}\intl_\t^\infty (\z\!-\!\t)^{\a -1} d\z \intl^\infty_{\z}\! f_0 (r)(r^2\!-\!\z^2)^{(n-3)/2}r dr\nonumber\\
&&=\!\frac{\sigma_{n-2}}{\Gam(\a)} \intl^\infty_{\t}\!\! f_0 (r) k(r,\t) dr, \quad k(r,\t)\!=\! r\!\intl_\t^r \!(\z\!-\!\t)^{\a -1} (r^2\!-\!\z^2)^{(n-3)/2} d\z. \nonumber\eea

If $n\ge 3$, then $(r^2-\z^2)^{(n-3)/2} \ge (r-\z)^{(n-3)/2} r^{(n-3)/2}$,
\[k(r,\t) \ge c\, (r-\t)^{(n-3)/2 +\a} r^{(n-1)/2},\] and
\[
(R_{+}^\a \tilde f)(\theta, -\t)\ge c\,  \intl^\infty_{\t}\!\! \frac{(r\!-\!\t)^{(n-3)/2 +\a}\,  r^{(n-1)/2}}{(2+r)^{n/p} \,\log (2+r)}\, dr \equiv \infty\]
if $p\ge n/(n-1+\a)$. Recall that the constant $c$ may be different at each occurrence. If $n=2$, then
\[(r^2\!-\!\z^2)^{-1/2} \ge (r\!-\!\z)^{-1/2} (2 r)^{-1/2} ,  \quad k(r,\t) \ge c\, (r\!-\!\t)^{\a-1/2} r^{1/2},\] and the conclusion is the same.
\end{proof}

 \begin{theorem}\label {o9iu2}   Let  $\,0<\a <1$,
\be\label{op98yoN} 1\le p < \frac{n}{n-1+\a}, \quad \mu = -n \left(1-\frac{2}{p}\right), \quad \nu= -\a-\frac {n-1}{p'}, \ee
$1/p +1/p' =1$.
We define
\[
u(x) = \frac {|x_n|^\nu}{(1+|x'|^2)^{\mu/2}},  \]
\be\label {space2} L^p_u(\bbr^n)=\{ f: ||f||_{p,u}\equiv || u(x) f||_p <\infty \}.\ee
Then
\[
||T_{+}^\a f||_{p,u}\le c \,||f||_{p},\qquad c =c\,(\a,n,p)=\const.\]
The assumptions (\ref{op98yoN}) are sharp.  There exists a function $\tilde f \in L^p(\bbr^n)$ with $p \ge n/(n-1+\a)$, such that $(T_{+}^\a \tilde f)(x) \equiv \infty$.
\end{theorem}
\begin{proof}
 Set $\nu=-\a-(n-1)/p'$ in (\ref{tae1}) and make use of Theorem \ref {o9iuN}, combined with Lemma \ref{ordi}.
\end{proof}

\subsection {Fractional integrals of smooth functions}
 It is convenient to write
  $T_{+}^\a f$ and $\stackrel{*}{T}\!{}_{+}^\a g$ in the form
\[ (T_{+}^\a f)(x', x_n)\!=\! (h^\a_{+} \ast T_{x'}f)(x_n),\quad (\stackrel{*}{T}\!{}_{+}^\a g)(x', x_n)\!=\! (h^\a_{-} \ast  T^*_{x'}g)(x_n),\]
or
\be\label {ted2.8}  (T_{+}^\a f)(x', x_n)\!=\! (I^\a_+ T_{x'} f)(x_n), \quad (\stackrel{*}{T}\!{}_{+}^\a g)(x', x_n)\!=\! (I^\a_-  T^*_{x'}g)(x_n);\ee
cf. (\ref{desds2}), (\ref{t2.8}).
  Similarly,
\be\label {tre2.3}
(R_{+}^\a  f)(\theta, t)= (I^\a_+ R_{\theta} f)(t).\ee
The following equalities can be easily checked for functions $\om \in S(\bbr)$:
\be\label {oqa}
I^\a_\pm I^\b_\pm\omega = I^{\a+\b}_\pm \om, \qquad Re \, \a>0, \quad Re \, \b >0; \ee
\be\label {oqa1}
(\pm d/dt)^{\ell} (I^{\a+\ell}_\pm \omega)(t)=  (I^{\a}_\pm \omega)(t), \qquad  Re \, \a>0, \quad \ell \in \bbn;\ee
\be\label {oqa2}
 (I^{\a}_\pm \omega)(t) = (\pm 1)^\ell (I^{\a+\ell}_\pm \om^{(\ell)}(t), \qquad   Re \, \a> 0, \quad \ell \in \bbn,\ee
 where $ \om^{(\ell)}(t)= (d/dt)^{\ell}\om (t)$.
  The equality (\ref{oqa2}) can be used for analytic continuation of $I^{\a}_\pm \om$ to the domain $Re \, \a> -\ell$. It follows that,
 $I^\a_\pm \omega$ extend to  $\a \in \bbc $ as entire functions of $\a$ with preservation of the properties (\ref{oqa})-(\ref{oqa2}).

These facts extend to the operators $T_{+}^{\a}$,   $\stackrel{*}{T}\!{}_{+}^\a$, and $R_{+}^{\a}$.
\begin{lemma} \label {sideT} Let  $f \!\in S (\rn)$.  ${}$\hfill

\noindent {\rm (i)} For each $x \in \rn$, $(T_{+}^{\,\a} f)(x)$
 extends as an entire function of $\a$, so that
\be\label {Dacz} \lim\limits_{\a \to 0} (T_{+}^{\,\a} f)(x) =(Tf)(x),\ee
where $(T f)(x)$ is the transversal Radon transform (\ref{bart}).

\noindent {\rm (ii)} For $Re \, \a>0$ and $ \ell \in \bbn$ the following relations hold:
%\be\label {oqart}
%I^\a_+ T_{+}^\b f = T_{+}^{\a+\b} f, \qquad Re \, \a>0, \quad Re \, \b >0; \ee
\be\label {oqa1rt}
(d/dx_n)^{\ell} (T_{+}^{\a+\ell} f)(x)=  (T_{+}^{\a} f)(x),\ee
\be\label {oqa2rt}
 (T_{+}^{\a}f)(x) =  (T_{+}^{\a+\ell} \partial_n^{\ell} f)(x).\ee
%The equality (\ref{oqa2rt}) can be used for realization of the analytic continuation of $T_{+}^{\a} f $   in the domain $Re \, \a> -\ell$.
%  with preservation of the properties (\ref{oqart})-(\ref{oqa2rt}).
\end{lemma}
\begin{proof}  By Lemma \ref{from}, $(Tf)(x', x_n)$ belongs to $S(\bbr)$ in the $x_n$-variable for each $x'\in \bbr^{n-1}$.
Hence, owing to the properties of the Riemann-Liouville fractional integrals, $(T_{+}^\a f)(x', x_n)\!=\! (I^\a_+ T_{x'}f)(x_n)$ extends as an entire function of $\a$ satisfying (\ref{Dacz}).
Other statements of the lemma  hold for the same reason.
\end{proof}

 An analogue of Lemma \ref{sideT} holds for $\stackrel{*}{T}\!{}_{+}^\a g$. For the fractional integral  $R_{+}^\a  f$ we have the following statement.
\begin{lemma} \label {sideTr} Let  $f \!\in S (\rn)$.  ${}$\hfill

\noindent {\rm (i)} For each $(\theta, t) \in Z_n$, $(R_{+}^{\,\a} f)(\theta, t)$
 extends as an entire function of $\a$, so that
\be\label {Dacr} \lim\limits_{\a \to 0} (R_{+}^{\,\a} f)(\theta, t) =(Rf)(\theta, t).\ee

\noindent {\rm (ii)} For $Re \, \a>0$ and $ \ell \in \bbn$ the following relations hold:
%\be\label {oqart}
%I^\a_+ T_{+}^\b f = T_{+}^{\a+\b} f, \qquad Re \, \a>0, \quad Re \, \b >0; \ee
\be\label {oqa1rtr}
(d/dt)^{\ell} (R_{+}^{\a+\ell} f)(\theta, t)=  (R_{+}^{\a} f)(\theta, t),\ee
\be\label {oqdd}
 (R_{+}^{\a}f)(\theta, t) =  (I^{\a+\ell}_+ R^{(\ell)}_\theta f)(t)), \ee
where
\bea
 (R^{(\ell)}_\theta f)(t)&=& (d/dt)^{\ell} (R f)(\theta, t)\nonumber\\
 \label{906clksu}  &=& \left \{
\begin{array} {ll} \! R\,[\Del^m f](\theta, t),&\mbox{ if $\ell =2m$},\\
\sum\limits_{k=1}^n \theta_k (R\,[\partial_{k} \Del^m f])(\theta, t),&\mbox{ if $\ell =2m +1$},\\
 \end{array}
\right.\eea
\end{lemma}
\begin{proof}  By Theorem \ref{mnub}, $(R f)(\theta, t)$ belongs to $S (\bbr)$ in the $t$-variable for each $\theta\in S^{n-1}$.
Hence $(R_{+}^{\a} f)(\theta, t)\!=\! (I^\a_+ R_{\theta} f)(t)$ extends as an entire function of $\a$ satisfying (\ref{Dacr}).
Other statements hold for the same reason. The differentiation formula (\ref{906clksu}) can be found in \cite[p. 3]{H11}, \cite[p. 135]{Ru15}.
\end{proof}

The following definition  will be used in the next sections.

\begin{definition}\label {exte} {\it For  $f\in S(\rn)$ and $\a \in \bbc $, we define  $(T_{+}^{\a}f)(x)$   and  $(R_{+}^{\a} f)(\theta, t)$ by the formulas
\bea\label {bft}
(T_{+}^{\a}f)(x)&=& a.c. \, (I^\a_+ T_{x'}f)(x_n),\\
\label {bftr}    (R_{+}^{\a} f)(\theta, t) &=&  a.c. \, (I^\a_+ R_{\theta} f)(t), \eea
where ``$a.c.$'' denotes analytic continuation. For each $\ell \in \bbn$, this analytic continuation can be realized in  the domain $Re \, \a> -\ell$ by (\ref{oqa2rt}) and (\ref{oqdd}), respectively.
Equivalently,
\bea \label {bft1}\qquad (T_{+}^{\a}f)(x)= (h^\a_{+} \ast T_{x'}f)(x_n)&=& (h^\a_{+} (y_n), (T_{x'}f)(x_n-y_n)), \\
\label {bft1r}(R_{+}^{\a}f)(\theta, t) = (h^\a_{+} \ast R_{\theta} f)(t)&=& (h^\a_{+} (s), (R_{\theta} f)(t-s)),\eea
where $h^\a_{+}$ is regarded as the $S'(\bbr)$-distribution.}
\end{definition}

We recall that $(T_{x'}f)(\cdot)$ and   $(R_{\theta} f)(\cdot)$  belong to $S(\bbr)$  and the right-hand sides of (\ref{bft1}) and (\ref{bft1r}) are
    infinitely differentiable tempered function of $x_n$ and $t$, respectively;  cf. e.g., \cite [p. 26, Lemma 2.5]{Es}, \cite [p. 84]{Vl}.

\subsubsection{ Fractional integrals in the Semyanistyi-Lizorkin space}

Let $\Phi (\rn)$ be the    Semyanistyi-Lizorkin space, and let $\Phi_0 (\bbr)$ be a similar space on the real line; see
Section \ref{orkin}.
It is known (see, e.g., \cite [Section 8.2] {SKM}) that the Riemann-Liouville operators $I^{\a}_\pm$, $\a\in \bbc$, are automorphisms of  $\Phi_0 (\bbr)$ and the following relations hold for the Fourier transforms:
\be\label {simil} (I^{\a}_\pm \omega)^{\wedge}(\t)= (\mp i\t)^{-\a} \hat \om (\t).\ee
Here the branch of the multi-valued power function is chosen so that
\be\label {mza}
(\mp i\t)^{-\a}=  \exp \left(-\a \log |\t| \pm \frac{\a \pi i}{2}\, \sgn \,\t \right ).\ee
Note also that the transversal Radon transform $T$ and its dual $T^*$ are automorphisms of $\Phi (\bbr^n)$; cf. \cite [Theorem 4.142]{Ru15}.
By (\ref{ted2.8}), it follows that the operators $T_{+}^\a$ and $\stackrel{*}{T}\!{}_{+}^\a$  are automorphisms of $\Phi (\bbr^n)$, too.

Given a function $f (x)\equiv f(x', x_n)$ on $\rn$, we denote by $\F_1$ and $\F_2$ the Fourier transforms of $f$ in the $x'$-variable and the $x_n$-variable, respectively.
Let also $I_2^{\lam} \om $, $\lam \in \bbc$, be the Riesz potential operator, which can be defined on functions $\om \in \Phi_0 (\bbr)$ by the formula
\be\label{Riesz}
(I_2^{\lam} \om)^{\wedge}(\t)= |\t|^{-\lam} \hat \om (\t).\ee

\begin{theorem}\label {pzp} {\rm (cf. \cite [Theorem 4.145]{Ru15})} If $f \in
\Phi (\bbr^n)$, then for  any complex $\a$ and $\b$,
\be\label{cab}
T_{+}^\a \stackrel{*}{T}\!{}_{+}^\b f =(2\pi)^{n-1} \, I_{+}^\a I_{-}^\b I_2^{n-1}f,\ee
\be\label{cabf}
\stackrel{*}{T}\!{}_{+}^\b T_{+}^\a f =(2\pi)^{n-1} \,  I_{-}^\b I_{+}^\a I_2^{n-1}f.\ee
 \end{theorem}
\begin{proof} Let $g=\tilde T_{-}^\b f $. Then $g \in \Phi (\bbr^n)$ and
(\ref{prslT}) yields
\be\label {oh}(\F_2 T_{+}^\a g)^\wedge(\xi_n)= (I^\a_+ T_{x'}g)^\wedge(\xi_n)=(- i\xi_n)^{-\a} \hat g (-x'\xi_n, \xi_n).\ee
Here
\bea \hat g (-x'\xi_n, \xi_n)&=&
(\F_1 \{\F_2 [g(y', \cdot)] (\xi_n) \}) (-x'\xi_n)\nonumber\\
&=& (\F_1 \{\F_2 [(I_{-}^\b  T^*  f)(y', \cdot)] (\xi_n) \}) (-x'\xi_n)\nonumber\\
&=& (\F_1 \{(i\xi_n)^{-\b}\F_2[(T^* f)(y', \cdot)] (\xi_n) \}) (-x'\xi_n)\nonumber\\
&=& (i\xi_n)^{-\b}  (\F_1 [\hat f (y'\xi_n, \xi_n)])(-x'\xi_n)\nonumber\\
&=& (i\xi_n)^{-\b} \intl_{\bbr^{n-1}} \hat f (y'\xi_n, \xi_n)\, e^{-i (x'\xi_n) \cdot y'}\, dy'\nonumber\\
&=&(i\xi_n)^{-\b} |\xi_n|^{1-n} \intl_{\bbr^{n-1}} \hat f (z', \xi_n)\, e^{-ix' \cdot z'}\, dz'\nonumber\\
&=& (2\pi)^{n-1} \,(i\xi_n)^{-\b} |\xi_n|^{1-n}\, (\F_2[f(y', \cdot)])(\xi_n).\nonumber\eea
 Hence, by (\ref{oh}),
\[
(\F_2 T_{+}^\a \stackrel{*}{T}\!{}_{+}^\b f )^\wedge(\xi_n)= (2\pi)^{n-1} \,(- i\xi_n)^{-\a} (i\xi_n)^{-\b} |\xi_n|^{1-n}\,
 (\F_2[f(y', \cdot)])(\xi_n).\]
 This gives  (\ref{cab}).  The proof of (\ref{cabf}) is similar.
\end{proof}

Setting $\a=\b=(1-n)/2$ and  $\a=\b=0$, we obtain the following corollary.
\begin{corollary} Let  $f \in \Phi (\bbr^n)$. Then
\be\label{ccab}
T_{+}^{(1-n)/2}  \stackrel{*}{T}\!{}_{+}^{(1-n)/2} f  =(2\pi)^{n-1} \,f;\ee
\be\label{ccab1}
T T^* f = T^* T f =(2\pi)^{n-1} \,I_2^{n-1}f.\ee
\end{corollary}

\begin{remark} The equality (\ref{ccab}) can be easily understood if we notice that   \[ (- i\xi_n)^{-\a} (i\xi_n)^{-\a} =|\xi_n|^{-2\a}.\]
Furthermore, it will be proved (see Lemma \ref {ing}) that (\ref{ccab}) extends to all $f \in L^2 (\bbr^n)$.
An analogue of the second equality in (\ref{ccab1}) for the Radon transform $R$ is well known and has the form
\be\label{krya} R^*Rf=c_n\,I^{n-1} f, \qquad c_n=2^{n-1}\pi^{n/2-1} \Gam (n/2), \ee
where $I^{n-1} f$ is the Riesz potential of order $n-1$ in $\rn$ and $R^*$ is the dual Radon transform, which averages $Rf$ over all hyperplanes passing through $x$; see, e.g., \cite [Proposition 4.37]{Ru15}.
\end{remark}

\section{$L^p$-$L^q$ estimates of $T_{+}^{\a}f$  and interpolation}

\subsection{Some preparations}

Let $f \in L^p (\rn)$. If $Re \,\a > 0$, then, by Theorem \ref{o9iu2},  $T_{+}^{\a}f$ exists a.e. as an absolutely convergent integral provided $1\le p < n/(n-1+ Re\,\a)$, and this bound is sharp.
If $Re \,\a \le  0$, the integral $T_{+}^{\a}f$ diverges. However, if  $f\in S(\rn)$, then $T_{+}^{\a}f$ can be defined by analytic continuation; see Definition \ref{exte}.

We wonder,  for which $\a$, $p$, and $q$  the operator $T_{+}^{\a}: S(\rn) \to C^\infty (\rn)$ extends as a linear bounded operator (we denote the extended operator by  $\T_{+}^{\a}$) mapping $L^p (\rn)$ to $L^q (\rn)$.

To answer this question, one may be tempted to make use of Stein's interpolation theorem for analytic families of operators \cite{Ste56, BS, Graf04,  SW}. However, the assumptions of this theorem and its proof are given in terms of simple functions, which are not smooth. It is unclear how to define our operators on such functions in the case $Re \,\a <0$, when the integrals are divergent, and check the validity of the assumptions of the theorem; cf. Definition (\ref {bft}), which is given  in terms  functions  belonging to $S(\rn)$.

The following   theorem, which is applicable to our situation, can be found in Grafakos \cite{Graf}.

Given $0< p_0, p_1 \le \infty$ and $0<  q_0, q_1 \le \infty$, let
\be\label {wzee1}
\frac{1}{p}= \frac{1-\theta}{p_0} + \frac{\theta}{p_1}, \qquad \frac{1}{q}= \frac{1-\theta}{q_0} + \frac{\theta}{q_1}; \qquad 0<\theta <1.\ee
Denote
\[ {\bf S}=\{z \in \bbc : 0<  Re\, z < 1\}, \qquad {\bf \bar S}=\{z \in \bbc : 0\le Re\, z \le 1\}.\]

\begin {theorem} \label {lanse1} \cite [Theorem 5.5.3]{Graf}
 Let $T_z$, $z \in  {\bf \bar S}$, be a family of linear operators  satisfying  the following conditions.

\vskip 0.2 truecm

 {\rm \bf(A)} For each $z \in  {\bf \bar S}$, the operator  $T_z$  maps  $C_c^\infty (\rn)$  to $L^1_{loc} (\rn)$.

\vskip 0.2 truecm

 {\rm \bf(B)}  For all $\vp, \psi \in C_c^\infty (\rn)$, the function
\be\label {wzws2as}
\A (z)=\intl_{\rn}  (T_z \vp)(x) \psi (x)\, dx \ee
is analytic in  ${\bf S}$ and continuous  on  ${\bf \bar S}$.

\vskip 0.2 truecm

 {\rm \bf(C)} There exist constants $\gam \in [0, \pi)$ and $s\in (1, \infty]$, such that for any $\vp \in C_c^\infty (\rn)$ and any compact subset $K\subset \rn$,
\be\label {wzws2as1}
\log ||T_z \vp||_{L^s (K)}\le C \, e^{\gam |Im \, z|}\ee
for all $z\in {\bf \bar S}$ and some constant $C=C(\vp, K)$.

\vskip 0.2 truecm

 {\rm \bf(D)} There exist constants  $B_0$, $B_1$, and  continuous functions $M_0 (\gam)$, $M_1 (\gam)$ satisfying
\be\label {wzed1}
M_0 (\gam) + M_1 (\gam) \le \exp (c\, e^{\t |\gam|})\ee
with some constants $c\ge 0$ and $0\le \t <\pi$, such that
 \be\label {wzed} ||T_{i\gam}f ||_{q_0} \le B_0 M_0 (\gam) ||f||_{p_0}, \qquad ||T_{1+i\gam}f||_{q_1} \le B_1 M_1 (\gam) ||f||_{p_1} \ee
for all $\gam \in \bbr$ and all  $f \in C_c^\infty (\rn)$.

\vskip 0.2 truecm

 Then  for all $f \in C_c^\infty (\rn)$,
\be\label {wzed2} ||T_{\theta}f ||_q  \le B_{\theta} M_{\theta} \,||f||_{p},\ee
where $B_{\theta} =B_0^{1-\theta} B_1^{\theta}$ and
\[ M_{\theta}\!=\!\exp \left \{ \frac{\sin (\pi\theta)}{2} \!\intl_{-\infty}^{\infty} \!\left [ \frac{\log M_0 (\gam)}{\cosh (\pi\gam) \!-\!\cos (\pi\theta)}\! + \!
\frac{\log M_1 (\gam)}{\cosh (\pi\gam)\! +\!\cos (\pi\theta)}\right ] d\gam \right \}.\]
\end{theorem}

Our aim is to apply Theorem \ref{lanse1} to the analytic family $\{T_+^\a\}_{\a \in \bbc}$.
First, we examine the cases $Re \,\a=(1-n)/2$ and  $Re \,\a=1$.

\begin{lemma} \label {ing} For the operator $T_+^{(1-n)/2 +i\gam}$, $\gam \in \bbr$, initially defined by (\ref{bft}),
there exists a unique linear bound extension
\be\label {llzza} \T_+^{(1-n)/2 +i\gam}: L^2 (\rn)\to L^2(\rn), \ee
such that for every $f\in L^2 (\rn)$,
\be\label{krccs}
||\T_+^{(1-n)/2 +i\gam} f||_2 \le c_n e^{|\gam|\pi /2}\, ||f||_2, \qquad c_n=  2^{n/2}  \pi^{(n-1)/2} .\ee
Moreover, in the case $\gam=0$, the operator $c_n^{-1}\T_+^{(1-n)/2}$ is unitary.
\end{lemma}
\begin{proof} We recall that the Fourier transform of the $S'$-distribution $h^\a_{+}$, $\a \in \bbc$, is computed as
\be\label {mzaX}
\hat h^\a_{+} (\xi_n)=(- i\xi_n)^{-\a}\equiv  \exp \Big(\!-\a\log |\xi_n| + \frac{\a \pi i}{2}\, \sgn \,\xi_n \Big );\ee
cf.   (\ref {mza}).
This formula can be found in different forms in many sources; see, e.g.,
 \cite[pp.  98, 137]{SKM}, \cite [Chapter II, Section 2.3]{GS1}, \cite [Chapter I, Section 2]{Es}.
  Suppose $Re \, \a \le 0$. This assumption guarantees absolute convergence of all integrals in our calculations below.
 By the Plancherel formula, (\ref{bft1}) and (\ref{prslT}) yield

\bea(T_{+}^{\a}f)(x)&=& \frac{1}{2\pi} \langle (h^\a_{+} (y_n))^{\wedge}, ((T_{x'}f)(x_n+y_n))^{\wedge}\rangle\nonumber\\
&=&\frac{1}{2\pi} \langle (- i\xi_n)^{-\a}, e^{-i x_n \xi_n }  (T_{x'}f)^{\wedge} (\xi_n)\rangle\nonumber\\
&=& \frac{1}{2\pi} \langle (- i\xi_n)^{-\a}, e^{-i x_n \xi_n} \hat f(-x'\xi_n, \xi_n)\rangle.\nonumber\\
&=& F^{-1} [(- i\xi_n)^{-\a} \hat f(-x'\xi_n, \xi_n)](x_n),\nonumber\eea
where $F^{-1}$ stands for the inverse Fourier transform in the last variable. Note that if  $\a_0 = Re\, \a$ and $\gam = Im \,\a$, then, by (\ref{mzaX}),
\[
|(- i\xi_n)^{-\a}|= |\xi_n|^{-\a_0} \exp \Big(\!- \frac{\gam \pi}{2}\, \sgn \,\xi_n \Big ).\]
Hence, by the Parseval equality,
\bea
&&\intl_{-\infty}^\infty |(T_{+}^{\a}f)(x', x_n)|^2 d x_n= \frac{1}{2\pi}\intl_{-\infty}^\infty |(- i\xi_n)^{-\a} \hat f(-x'\xi_n, \xi_n)|^2  d\xi_n\nonumber\\
&&=\frac{1}{2\pi}\intl_{-\infty}^\infty  |\xi_n|^{-2\a_0} \exp (-\gam \pi\, \sgn \,\xi_n)\,  |\hat f(-x'\xi_n, \xi_n)|^2 d\xi_n\nonumber\\
&&=\frac{ e^{-\gam \pi}}{2\pi}\intl_{0}^\infty r^{-2\a_0}  |\hat f(-x'r, r)|^2 dr +  \frac{ e^{\gam \pi}}{2\pi}\intl_{0}^\infty r^{-2\a_0}  |\hat f(x'r, -r)|^2 dr.\nonumber\eea
Now integration in $x'$  gives
\[||T_{+}^{\a}f||_2^2=  \frac{\ch  \gam \pi}{\pi} \intl_{\rn} |y|^{-2\a_0 +1-n} |\hat f(y)|^2 \, dy.\]
If $\a_0= (1-n)/2$, then
\be\label {eref}   ||T_{+}^{(1-n)/2 +i\gam}f||_2^2=  \frac{\ch  \gam \pi}{\pi} ||\hat f||_2^2= (2\pi)^n \frac{\ch  \gam \pi}{\pi} ||f||_2^2.\ee
Because $|\ch  \gam \pi|\le e^{|\gam| \pi}$, the latter gives (\ref{krccs}) for $f\in S(\rn)$. Now the statements of the lemma follow  from the density of $ S(\rn)$ in $L^2(\rn)$.
If $\gam =0$ then (\ref{eref}) gives
\[   ||T_{+}^{(1-n)/2 +i\gam}f||_2=  c_n\,  ||f||_2, \qquad c_n= 2^{n/2} \pi^{(n-1)/2}.\]
Hence the operator $c_n^{-1}\T^{(1-n)/2}_{+}:  L^2(\rn) \to L^2(\rn)$ is an isometry.
To show that it is  unitary, it remains to note that it is ``onto''. Indeed, by (\ref{ccab}), if $ f \in \Phi(\rn)$, then
\[ f= (2\pi)^{1-n} T_{+}^{(1-n)/2}   \stackrel{*}{T}\!{}_{+}^{(1-n)/2} f.\]
where, by (\ref{tocd}),
\[
|| \stackrel{*}{T}\!{}_{+}^{(1-n)/2} f||_2= || \stackrel{*}{T}\!{}_{+}^{(1-n)/2} \dot f||_2 =c_n\,  ||\dot f||_2 = c_n\,  ||f||_2.\]
Because $\Phi(\rn)$ is dense in  $L^2(\rn)$  (see  Proposition \ref{izsh}), the result follows.
\end{proof}

\begin{remark} The last statement of Lemma \ref{ing} about unitarity of the operator $c_n^{-1}\T_+^{(1-n)/2}$ is not needed for interpolation. However, it is of interest on its own right. In the close situation, operators of this kind occur in inverse problems for the wave equation; cf. \cite [Proposition 1]{BK}.
\end{remark}

Now we consider the case
 $\a=1 +i\gam$, $\gam \in \bbr$, when
\[
(T^{1 +i\gam}_{+} f)(x)  \frac{1}{\Gam (1 +i\gam )} \intl_{\rn} (x_n -y_n)_{+}^{i\gam}\, f(y', y_n + x' \cdot y')\, dy.\]
Note that
\[
|\Gam (1 +i\gam )|^2 =\frac{\pi \gam}{\sinh (\pi \gam)};\]
see, e.g., \cite{AS}. This gives  the following statement.

\begin{lemma} \label {ingq} For any $\gam \in \bbr$ and $f \in L^1(\rn)$,
\be\label{krtts}
||T^{1 +i\gam}_{+} f||_{\infty} \le e^{\pi |\gam|/2}\,  ||f||_1.   \ee
\end{lemma}

\subsection{The main theorem}

Below we apply Theorem \ref{lanse1} to the operator family $\{T_+^\a\}$.

\begin {theorem} \label {lanseT} Suppose $1\le p,q \le \infty$, $\a_0=Re\, \a$.  The operator $T_+^\a$, initially defined by (\ref{bft}) on functions $f \in S (\rn)$, extends as a linear bounded operator $\T_+^\a$ from $L^p (\rn)\!$ to $L^q (\rn)\!$ if and only if
\be\label {wz2asTT}
\frac{1-n}{2} \le \a_0 \le 1, \qquad p=\frac{n+1}{n + \a_0}, \qquad q=\frac{n+1}{1-\a_0}.\ee
\end{theorem}
\begin{proof}

The proof consists of two parts  related to the ``if'' statement and the ``only if'' statement, respectively.

\vskip 0.2 truecm

\noindent {\bf Part I.}
We define
\be\label {wzeda}  T_z= T_{+}^{\,\a (z)}, \qquad \a(z)=\frac{1+n}{2} z + \frac{1-n}{2}.\ee
Then $0\le Re\, z \le 1$ corresponds to $(1-n)/2 \le Re\, \a \le 1$. In our case,
\[p_0=q_0=2, \qquad p_1=1, \qquad q_1=\infty.\]
If for real  $\a \in [(1-n)/2, 1]$ we set $\a=((1+n)/2) \theta + (1-n)/2$,  then $\theta=(2\a+n-1)/(n+1)$, and (\ref{wzee1}) yields
\[ 1/p= (\a +n)/(n+1), \qquad 1/q=(1-\a)/(n+1).\]
The latter agrees with (\ref{wz2asTT}).

To apply  Theorem \ref {lanse1}, we must show, step by step,  that the operator family $T_z= T_{+}^{\,\a (z)}$ meets  the conditions  {\rm \bf(A)}-{\rm \bf(D)}. This important technical part of the proof relies on the material of the previous sections.

\vskip 0.2 truecm

\noindent {\rm \bf(A)}
Let $\vp \in C_c^\infty (\rn)$. Fix any integer $\ell > (n-1)/2$. Then for all $(1-n)/2 \le Re\, \a  \le 1$,
\[
(T_{+}^{\a} \vp)(x)=  (T_{+}^{\,\a +\ell} \partial_n^\ell  \vp)(x)=\frac{1}{\Gam (\a +\ell)}\intl_0^\infty \!s ^{\a+\ell -1} A^{(\ell)}_{x} (s)\,ds,\]
where, by (\ref{99jh1}),
\[
 A^{(\ell)}_{x} (s)=\intl_{\bbr^{n-1}} \!\!\!  (\partial_n^\ell  \vp) (y', x_n - s + x' \cdot y') \, dy'= (T\partial_n^\ell \vp)(x', x_n-s)  .\]
Hence, by (\ref{pari}), for sufficiently large $m\in \bbn$ there is a constant $c_m$ such that
\[|A^{(\ell)}_{x} (s)| \le c_m \,\frac{(1+|x'|^2)^m}{(1+|x_n-s|)^m}\le c_m \,\frac{(1+|x'|^2)^m \,  (1+|x_n|)^m}{(1+s)^m}.\]
 It follows that
 \bea
 &&|(T_{+}^{\a} \vp)(x)| \le \frac{ c_m \,(1+|x'|^2)^m \,  (1+|x_n|)^m}{|\Gam (\a +\ell)|}\intl_0^\infty \!\frac{s ^{\a+\ell -1}}{(1+s)^m}\, ds\nonumber\\
\label {fint} &&=
 \frac{\tilde c_m}{|\Gam (\a \! +\!\ell)|} \,(1\!+\!|x'|^2)^m   (1\!+\!|x_n|)^m, \qquad \tilde c_m \!=\!\const,\eea
if $m$ is large enough. Hence   $T_{+}^{\a} \vp \in  L^1_{loc} (\rn)$ for all $(1-n)/2 \le Re\, \a  \le 1$.

\vskip 0.2 truecm

\noindent {\rm \bf(B)}  For any $\vp, \psi \in C_c^\infty (\rn)$, we have
\[
\A (z)=\intl_{\rn}  (T_z \vp)(x) \psi (x)\, dx=\intl_{{\rm supp} (\psi)} \!\!\!F (x,z)\,dx,\]
where
\[ F (x,z)= (T_z \vp)(x) \psi (x)= (T_{+}^{\,\a (z)} \vp)(x)\psi (x)\]
is an entire function of $z$  (see Lemma \ref{sideT}(i)), which is locally integrable in the $x$-variable for each $z$ in any compact subset of $\bf \bar S$; see {\rm \bf(A)}.
 Using known facts about analyticity of functions represented by integrals (see, e.g., \cite [Lemma 1.3]{Ru15}), we obtain  {\rm \bf(B)}.

\vskip 0.2 truecm

\noindent {\rm \bf(C)}
We make use of the (\ref {fint}), which gives
\be\label{kanzz1}
||T_{+}^{\,\a} \vp||_{L^\infty (K)}\le \frac{C}{|\Gam (\a +\ell)|}, \qquad C=C (K, \vp).\ee

Now, let us revert to the notation of the interpolation theorem.
If $\a=\a (z)$ and $T_z \vp =T_{+}^{\,\a (z)} \vp$, then for all $0\le Re\, z \le 1$,  (\ref{kanzz1})  gives
\[
 ||T_z \vp||_{L^\infty (K)}\le \frac{C}{|\Gam (\z)|}, \qquad \z=\frac{1+n}{2} z + \frac{1-n}{2} +\ell.\]
If $z=x+iy$, then $\z=a+ib$, where
\[
 a= \frac{1+n}{2}\, x + \frac{1-n}{2} +\ell, \qquad b=\frac{1+n}{2} \,y; \qquad 0\le x\le 1.\]
It is known that if  $a_1\le a\le a_2$ and $|b| \to \infty$, then
\be\label {quad}
|\Gam (\z)|= \sqrt {2\pi} \, |b|^{a-1/2}\,  e^{-\pi |b|/2}\,  [1 +O (1/|b|)],\ee
where the constant implied by $O$ depends only on $a_1$ and $a_2$; see, e.g., \cite [Corollary 1.4.4]{AAR}.
Taking into account that $1/\Gam (\z)$ is an entire function and using (\ref{quad}), after simple calculations we obtain an estimate of the form
\[ \log ||T_z \vp||_{L^{\infty} (K)}\le \left \{
\begin{array} {ll} \! c_1&\mbox{ if $|y|\le 10$},\\
c_2 +c_3 \log |y| +c_4 |y|&\mbox{ if $|y|\ge 10$},\\
 \end{array}
\right.\]
for some positive constants $c_1, c_2, c_3, c_4$.
This estimate yields (\ref {wzws2as1})  for any $\gam >0$. Thus
 verification of the assumption {\rm \bf(C)} in Theorem \ref {lanse1} is complete.

\vskip 0.2 truecm

\noindent {\rm \bf(D)}
Let us check (\ref{wzed1}). By (\ref{wzeda}),
\[  T_{i\gam}=T_{+}^{(1-n)/2 +i\gam (1+n)/2}, \qquad  T_{1+i\gam}=T_{+}^{1 +i\gam (1+n)/2}.\]
Hence the results of Lemmas \ref{ing} and \ref{ingq} can be stated as
\[ ||T_{i\gam}\vp ||_{2} \le B_0 M_0 (\gam) ||\vp||_{2}, \qquad ||T_{1+i\gam} \vp||_{\infty} \le B_1 M_1 (\gam) ||\vp||_{1}, \]
\[ M_0 (\gam)=M_1 (\gam)=\exp (\pi (1+n)|\gam|/4),\]
with some constants $B_0, B_1$. These estimates give (\ref{wzed1}).

Thus, by Theorem \ref {lanse1}, the ``if'' part of Theorem \ref{lanseT} is proved for $T^\a_{+}$ when $\a$ is real.
If $\a$ has a nonzero  imaginary part, say, $t$, the above reasoning can be repeated almost verbatim if we re-define $T_z$ in (\ref{wzeda})  by setting $T_z= T_+^{\,\a (z+it)}$.

\vskip 0.2 truecm

\noindent {\bf Part II.}
  Let us prove the ``only if'' part.
First we show that the left bound $Re \, \a= (1-n)/2$ in (\ref{wz2asTT}) is sharp.
Suppose the contrary,  assuming for simplicity that $\a$ is real.  Then there is  a triple $(p_0, q_0, \a_0)$ with $1\le p_0< q_0 \le \infty$ and $\a_0<(1-n)/2$, such that   $T^{\a_0}_{+}$  is bounded from  from $L^{p_0}(\rn)$ to $L^{q_0}(\rn)$.
Interpolating the  triples  $(p_0, q_0, \a_0)$ and $(1, \infty, 1)$, as we did above, we conclude that for any $\a \in [\a_0, 1]$, the operator $T^\a_{+}$ is bounded   from $L^{p_\a}(\rn)$ to $L^{q_\a}(\rn)$, where
\[ \frac{1}{p_\a}= \frac{\a -\a_0}{ 1 -\a_0} + \frac{1-\a }{ p_0 (1 -\a_0)}, \qquad  \frac{1}{q_\a}=  \frac{1-\a }{ q_0 (1 -\a_0)}.\]
 In particular, for $\a= (1-n)/2$ and $\a_0=(1-n)/2 -\e$, $0<\e< (1+n)/2$, we obtain
\[ p_{(1-n)/2}=\frac{2(1+n)}{1+n -2\e}, \qquad q_{(1-n)/2}=\frac{2(1+n)}{1+n +2\e}.\]
The latter agrees with the known $L^2$-$L^2$ boundedness of $T^{(1-n)/2}_{+}$ only if $\e=0$, that is, $\a_0=(1-n)/2$.

The necessity of $p$ and $q$ in (\ref{wz2asTT}) can be proved using the  scaling argument.
 Specifically, suppose, for example, that $\a$ is real.   Abusing notation, let $\lam =(\lam_1, \lam_2)$, $\lam_1 >0$, $\lam_2 >0$, and denote
\[
(A_\lam f)(x)=f (\lam_1 x', \lam_2 x_n), \qquad (B_\lam F)(x)=\frac{\lam_1^{1-n}}{\lam_2^\a } F\left (\frac{\lam_2}{\lam_1}\, x', \lam_2 x_n\right ). \]
Then $T_{+}^\a A_\lam f = B_\lam T_{+}^\a f$. This equality is straightforward if $\a>0$ and extends to all $\a\in\bbc$ by analyticity.
We have
\[ ||A_\lam f||_p=\lam_1^{(1-n)/p}\lam_2^{-1/p} ||f||_p, \qquad ||B_\lam F||_q=\lam_1^{1-n+(n-1)/q}\lam_2^{-\a-n/q} ||F||_q.\]
If $|| T^\a_{+} f||_q \le c \,||f||_p$ is true for all $f\in L^p$, then it is true for $A_\lam f$, that is, $||T^\a_{+} A_\lam f||_{q}\le c \,||A_\lam f||_{p}$ or
$||B_\lam T_{+}^\a f||_{q}\le c \,||A_\lam f||_{p}$.
 The latter is equivalent to
\[
\lam_1^{1-n+(n-1)/q}\lam_2^{-\a-n/q} ||T_{+}^\a f||_q   \le c \,\lam_1^{(1-n)/p}\lam_2^{-1/p} ||f||_p.\]
Assuming that $\lam_1 $ and $\lam_2 $ tend to zero and to infinity, we conclude that the last inequality is possible only if
\be\label {hghfv}
p=\frac{n+1}{n + \a}, \qquad q=\frac{n+1}{1- \a}.\ee
The above  reasoning  shows that the right bound $Re \,\a= 1$ is sharp, too.

Now, the proof of  Theorem \ref{lanseT} is complete.
\end{proof}

\begin{theorem}\label {o9iu4}  Let   $p$ and $q$  be defined by
\be\label {wz2asm}
 p=\frac{n+1}{n + Re \, \a}, \qquad q=\frac{n+1}{1- Re \, \a},\ee
  and let $f \in L^p (\rn)$.  In the cases $0<Re\, \a <1$ and  $\a=0$, the $L^q$-functions $(\T_{+}^{\,\a} f)(x)$ and $(\T f)(x)=(\T_{+}^{\,\a} f)(x)|_{\a =0}$, the existence of which is guaranteed by Theorem \ref{lanseT},
coincide a.e.  with  $(T_{+}^{\,\a} f)(x)$ and $(T f)(x)$, respectively.
\end{theorem}
\begin{proof} In the case $0<Re\, \a <1$, it suffices to assume that $\a$  is real. Because
\[p=(n+1)/(n + \a) < n/(n-1 +\a),\]
the operator $ T_{+}^{\,\a}$ obeys the conditions of Theorem \ref{o9iu2}.
 Thus, by  Theorems  \ref{lanseT} and \ref{o9iu2}, we have two linear bounded operators
\[\T_{+}^{\,\a}:   L^p (\rn) \to L^q (\rn), \qquad  T_{+}^{\,\a}:   L^p (\rn) \to L^p_u(\bbr^n),\]
where $u(x) = |x_n|^\nu (1+|x'|^2)^{-\mu/2}$.
Hence both operators are bounded from  $L^p (\rn)$ to $L^1 (K)$ for any compact set $K \subset \rn$ away from the hyperplane  $x_n=0$. Let $\{f_k\} \subset S(\rn)$ be a sequence approximating $f$ in the $L^p$-norm.
Because $\T_{+}^{\,\a}$ was defined as an extension of $T_{+}^{\,\a}$ from $S(\rn)$,     $(\T_{+}^{\,\a} f_k)(x)= (T_{+}^{\,\a} f_k)(x)$. Then, by linearity,
\[
\T_{+}^{\,\a} f - T_{+}^{\,\a} f = \T_{+}^{\,\a} (f -f_k) - T_{+}^{\,\a} (f -f_k),\]
and therefore
\bea
||\T_{+}^{\,\a} f - T_{+}^{\,\a} f||_{L^1 (K)} &\le& ||\T_{+}^{\,\a} (f -f_k)||_{L^1 (K)} +||T_{+}^{\,\a} (f -f_k)||_{L^1 (K)}\nonumber\\
&\le& c_1 ||f -f_k||_p + c_2  ||f -f_k||_p\nonumber\eea
 for some constants $c_1$ and $c_2$ depending on $K$. Assuming $k\to \infty$, we obtain $||\T_{+}^{\,\a} f - T_{+}^{\,\a} f||_{L^1 (K)}=0$. This gives $(\T_{+}^{\,\a} f)(x)=(T_{+}^{\,\a} f)(x)$ for almost all
 $x\in K$, and therefore for almost all  $x\in \rn$.

If $\a=0$, a similar  reasoning  relies on   Theorems  \ref{lanseT} and \ref {o9iu2t}.
\end{proof}

\section {$L^p$-$L^q$ estimates of $R_{+}^{\a}f$ }

Theorem \ref{lanseT} yields the following result for the operator $R_+^\a$.
\begin {theorem} \label {lanseTR} Suppose $1\le p,q \le \infty$, $\a_0=Re\, \a$.  The operator $R_+^\a$, initially defined by (\ref{bftr}) on functions $f \in S (\rn)$, extends as a linear bounded operator $\R_+^\a$ from $L^p (\rn)\!$ to $L^q (Z_n)\!$ if
\be\label {wz2asr}
\frac{1-n}{2} \le \a_0 \le 1, \qquad p=\frac{n+1}{n + \a_0}, \qquad q=\frac{n+1}{1-\a_0}.\ee
\end{theorem}
\begin{proof}  We recall that for $Re\, \a > 0$,  (\ref{4A5i}) yields
\be\label {mmzq}
 (R_{+}^\a f)(\theta, t)=  |\theta_n |^{\a-1}(T_{+}^\a f) \left (-\frac{\theta'}{\theta_n},\frac{t}{\theta_n}\right)), \qquad  \theta_n \neq 0.\ee
For $f \in S(\rn)$, let
\be\label {osw}
u(\theta, t)= a.c. \, (R_+^\a f)(\theta, t) \quad \text {\rm and} \quad v(x)= a.c. \, (T_+^\a f)(x)\ee
 be analytic continuations of the integrals $R_+^\a f$ and  $T_+^\a f$  from the domain $Re\, \a > 0$ to the entire complex plane. By  Lemmas \ref{sideT} and \ref{sideTr},
 these analytic
 continuations are well defined and represent smooth functions.
 Hence (\ref{mmzq}) extends analytically to all $\a \in \bbc$ and we get
 \[
 u(\theta, t)=  |\theta_n |^{\a-1} v\left (-\frac{\theta'}{\theta_n},\frac{t}{\theta_n}\right), \qquad  \theta_n \neq 0, \quad \a \in \bbc.\]
As in the proof of (\ref{tae}), for any $q\ge 1$ we have
\bea
(||u||_q^{\sim})^q &=& \intl_{\bbr} dt \intl_{S^{n-1}}  \left | |\theta_n |^{\a-1}  v\left (-\frac{\theta'}{\theta_n},\frac{t}{\theta_n}\right)\right |^q d\theta\nonumber\\
 &=&2\intl_{\bbr^{n}} \!\!\frac{|v(x)|^{q}}{(1\!+\!|x'|^2)^{(n+(\a_0 -1)q +1)/2}} \,dx. \nonumber\eea
In particular, if
 $q= (n+1)/(1-  \a_0)$, as in (\ref{wz2asr}), then  $||u||_{q}^{\sim}=2^{1/q} ||v||_{q}$. In other words, for any $f \in S(\rn)$ we have
 \be\label {rosw}
 ||a.c. \, R_+^\a f||_{q}^{\sim}=2^{1/q} ||a.c. \, T_+^\a f||_{q}= 2^{1/q} ||\T_+^\a f||_{q}.\ee
By Definition \ref {exte}   and  Theorem  \ref{lanseT}, it follows that  extension of (\ref{rosw}) to all $f \in L^p(\rn)$ gives
\[
 ||\R_+^\a f||_{q}^{\sim}\le 2^{1/q} c\, ||f||_{p},  \qquad  p=(n+1)/(n + \a_0),\]
 as desired.
\end{proof}

We observe that unlike Theorem \ref{lanseT},  Theorem \ref {lanseTR} does not include the  ``only if'' part. As we shall see below,  the values for $p$ and $q$ may differ from those in (\ref{wz2asr}).

\subsection{The Oberlin-Stein theorem}

The following statement, due to  Oberlin and Stein, is contained in Theorem 1 of the paper \cite{OS} (set $q=r$ in this theorem). Below we obtain it as a consequence of  our result for $Tf$.

\begin{theorem}\label {OST2} Let $1\le p, q\le \infty$.
For $n\ge 2$,
 the  inequality
\be\label {tein}
\|Rf\|^{\sim}_q \le c\,  \| f \|_p, \qquad c=\const,\ee
 holds
if and only if $1\le p\le (n+1)/n$ and  $1/q=n/p -n+1$.
\end{theorem}
\begin{proof}   First, note that  $(n+1)/n <n/(n-1)$, and therefore,  by Theorem \ref{lsu},
  both $Rf$ and $Tf$   exist a.e.  Further, if we make use of   (\ref{tae}) with $\nu=0$ and $p$ replaced by  $q=n+1$, we obtain
\[
\intl_{Z_n} |(Rf)(\theta, t)|^{n+1} dt d\theta= 2 \intl_{\bbr^{n}}  |(T f) (x)|^{n+1} dx.\]
 By Theorem \ref {lanseTR}, this gives the endpoint estimate in (\ref{tein}). The result for all $1\le p\le (n+1)/n$ then follows by interpolation, taking into account that
 $\|Rf\|^{\sim}_1=\sig_{n-1} ||f||_1$ by (\ref{vali}).

The necessity of the relations $1/q\!=\!n/p-n+1$ and  $p\!\le\! (n+1)/n$ was justified in \cite{OS}. For the sake of completeness,  we perform this justification  in  detail  in  a slightly different way, using the scaling argument.
Denote $f_\lam (x)=f(\lam x)$, $\lam >0$. Then $(R f_\lam) (\theta, t)= \lam^{1-n} (R f) (\theta, \lam t)$,
\[||f_\lam||_p= \lam^{-n/p} ||f||_p,\qquad || Rf_\lam ||^{\sim}_q= \lam^{1-n-1/q} || Rf||^{\sim}_q.\]
 If $|| Rf||^{\sim}_q\le c\, ||f||_p$ for all $f\in L^p (\rn)$, then, replacing $f$ by $f_\lam$, we obtain
\[\lam^{1-n-1/q} || Rf||^{\sim}_q \le \lam^{-n/p} ||f||_p,\]
which gives
\be\label {ordwa} 1-n-1/q =-n/p.\ee

To prove the necessity of the  bound $p\le (n+1)/n$, we set  $\tilde f_\lam (x)= f(\lam x', x_n)$. Then, as above,  $||\tilde f_\lam||_p= \lam^{(1-n)/p} ||f||_p$, and (\ref{tae})  (with  $p$ replaced by  $q$) yields
\[
(|| R\tilde f_\lam  ||^{\sim}_q)^q=2\intl_{\bbr^{n}} \!\!\frac{ |(T\tilde f_\lam) (x)|^q}{(1\!+\!|x'|^2)^{(n-q+1)/2}} dx.\]
Note that $(T\tilde f_\lam) (x)=\lam^{1-n} (T f) (x'/\lam, x_n)$.
Hence, changing variables, we obtain
\[
(|| R\tilde f_\lam  ||^{\sim}_q)^q= 2 \lam^{(n-1)(1-q)} \intl_{\bbr^{n}} \!\!\frac{ |(Tf) (x', x_n)|^q}{(1+|\lam x'|^2)^{(n-q+1)/2}} dx.\]
If $|| Rf||^{\sim}_q\le c\, ||f||_p$ for all $f\in L^p (\rn)$, then, setting $f=\tilde f_\lam$, we obtain
\[\lam^{(n-1)(1-q)/q} \,\Bigg (2\intl_{\bbr^{n}} \!\!\frac{ |(Tf) (x', x_n)|^q}{(1+|\lam x'|^2)^{(n-q+1)/2}} dx\Bigg )^{1/q} \le \lam^{(1-n)/p} ||f||_p,\]
or
\be\label{xst} \lam^{s} A^{1/q}(\lam) \le  ||f||_p, \qquad s=(n-1)(1-q)/q+ (n-1)/p,\ee
where  $A(\lam)$ stands for the expression in brackets. Let us pass to the limit in (\ref{xst}) as   $\lam \to 0$ assuming   $f$ to be  good enough.
If $s <0$, then the left-hand side of (\ref{xst}) tends to infinity. Hence, necessarily,  $s \ge 0$, which is equivalent to $q\le p'$. Combining the last inequality with (\ref{ordwa}), we obtain $p\le (n+1)/n$.
\end{proof}

\subsection{The Hardy-Littlewood-Sobolev theorem for $R^\a_+$}

The next theorem resembles the celebrated Hardy-Littlewood-Sobolev theorem for Riesz potentials; cf. \cite [Chapter V, Section 1.2]{Ste}, \cite [Theorem 0.3.2] {Sog}, \cite [p. 189] {Sog1}.
We use the notation  $m  \{\cdot\}$ for  the Lebesgue measure of the corresponding set and assume, for simplicity, that $\a$ is real-valued.
\begin{theorem}\label {OST2a}
Let $n\ge 2$, $0<\a<1$, $1\le p\le \infty$; $\, 1/p+ 1/p'=1$.

\vskip 0.2 truecm

\noindent {\rm (i)}
 If $p=1$, then   $R_+^\a$ is an operator of weak $(1,q)$-type with $ 1/q= 1-\a$, that is,
 \be\label{kde}
m \{ (\theta, t):  |(R^\a_+ f) (\theta, t)|>\lam \} \le c \left (\frac{||f||_1}{\lam}\right )^{1/(1-\a)} \text{for all $\;\lam >0$}. \ee

\vskip 0.2 truecm

\noindent {\rm (ii)} The  inequality
\be\label {teina}
\|R^\a_+f\|^{\sim}_q \le c\,  \| f \|_p, \qquad c=\const,\ee
 holds if and only if
 \be\label {law} 1< p \le p_\a, \qquad p_\a=\frac{n+1}{n + \a}, \qquad \frac{1}{q} = 1-\a-\frac{n}{p'}.\ee

 \end{theorem}
\begin{proof} The proof consists of three steps. First we prove {\rm (i)}, then the ``if'' part of {\rm (ii)}, and then the ``only if'' part of {\rm (ii)}.

\vskip 0.2 truecm

\noindent { STEP I.} Let us prove {\rm (i)}.  We  proceed as in \cite [Chapter V, Section 1.3]{Ste} with minor changes related to the specifics of our object.  It suffices to  prove  (\ref{kde}) for $||f||_1=1$. Indeed, if for such an $f$,
\[
m \{ (\theta, t):  |(R^\a_+ f) (\theta, t)|>\lam \} \le c \left (\frac{1}{\lam}\right )^{q}  \quad \text{for all $\;\lam >0$}, \]
then, for arbitrary  $f\in L^1 (\rn)$ with  $||f||_1\neq 0$ and $\lam$ replaced by $\lam/||f||_1$ we have
\[
m \{ (\theta, t):  |(R^\a_+ [f/||f||_1]) (\theta, t)|>\lam/||f||_1 \} \le c \left (\frac{||f||_1}{\lam}\right )^{q}. \]
The latter coincides with (\ref{kde})  by dilation argument.

Let $\mu$ be a fixed positive constant to be specified later. We  set
\[
(R_{+}^\a f)(\theta, t)\!=\!\frac{1}{\Gam (\a)}\intl_0^\infty \!\eta^{\a -1} (R_\theta f)(t\!-\!\eta) d \eta \!= \! k_1 \ast R_\theta f + k_\infty \ast R_\theta f,\]
where
\[
k_1 (\eta)=\frac{1}{\Gam (\a)}\left\{ \!
 \begin{array} {ll} \eta^{\a -1} & \mbox{if $ \eta \le \mu,$}\\
 0 & \mbox{if $\eta >\mu; $}\\
  \end{array}
\right.\qquad
k_\infty (\eta)=\frac{1}{\Gam (\a)}\left\{ \!
 \begin{array} {ll} 0 & \mbox{if $ \eta \le \mu,$}\\
 \eta^{\a -1} & \mbox{if $\eta >\mu. $}\\
  \end{array}
\right.\]
As in  \cite [Chapter V, Section 1.3]{Ste}, we  replace $\lam$ by $2\lam$ in the left-hand side of (\ref{kde}), which gives
\bea
&& m \{ (\theta, t):  |(R^\a_+ f) (\theta, t)|>2 \lam \} \nonumber\\
&&\le m \{ (\theta, t):  |(k_1 \ast R_\theta f)(t) |> \lam \} + m \{ (\theta, t):  |(k_\infty \ast R_\theta f)(t) |> \lam \}\nonumber\\
&&=m_1\{\cdot \} + m_\infty \{\cdot \}. \nonumber\eea
Thus it suffices to estimate $m_1\{\cdot \}$ and $m_\infty \{\cdot \}$.  Taking into account (\ref{vali}) and the assumption $||f||_1=1$,  we have
\bea
 ||k_1 \ast R_\theta f||^{\sim}_1&=&\frac{1}{\Gam (\a)}\intl_{Z_n}\Bigg | \intl_0^\mu \eta^{\a -1} (R_\theta f)(t\!-\!\eta) d \eta \Bigg | d\theta dt\nonumber\\
 &\le& \frac{\mu^\a}{\Gam (\a +1)}  ||R_\theta f||^{\sim}_1=\frac{\mu^\a \sig_{n-1}}{\Gam (\a +1)} ||f||_1=c_1 \mu^\a. \nonumber\eea
However, $ ||k_1 \ast R_\theta f||^{\sim}_1 \ge \lam m_1\{\cdot \}$. Hence
\be\label{mzar}
m_1\{\cdot \} \le c_1\, \frac{\mu^\a}{\lam}.\ee
Furthermore,  by (\ref{vali}),
\bea
||k_\infty \ast R_\theta f||^{\sim}_\infty \!\!&\le&\!\! \frac{1}{\Gam (\a)} \Big \| \intl_\mu ^\infty \!\eta^{\a -1} (R_\theta f)(t\!-\!\eta) d \eta \Big \|^{\sim}_\infty\nonumber\\
&\le&\!\! \frac{\mu^{\a -1}}{\Gam (\a)} \Big \| \intl_{-\infty} ^\infty\! (R_\theta |f|)(\eta) d\eta \Big \|_{L^\infty (S^{n-1})}\!\!= \frac{\mu^{\a -1}}{\Gam (\a)} ||f||_1\!=\!c_2 \mu^{\a -1}.\nonumber\eea
Now we choose $\mu$ so that $c_2 \mu^{\a -1}=\lam$, i.e., $\mu=(c_2/\lam)^{1/(1-\a)}$. Then $m_\infty \{\cdot \}=0$ and, by  (\ref{mzar}),
\[
 m \{ (\theta, t):  |(R^\a_+ f) (\theta, t)|>2 \lam \}=m_1\{\cdot \} \le c_3\, \lam^{-1/(1-a)}= c_3 \left (\frac{||f||_1}{\lam}\right )^{q}.\]
This gives  (\ref{kde}).

\vskip 0.2 truecm

\noindent { STEP II.}  Let us prove the ``if'' part of  {\rm (ii)}. It relies on the connection between $R^\a_+f$ and $T^\a_+f$.
 Recall that  by Theorems \ref{o9iuN} and \ref{o9iu2},
  both integrals   exist a.e. provided  $p< n/(n-1+\a)$. We meet this condition because \[p_\a=(n+1)/(1+ \a)< n/(n-1+\a).\]
By   (\ref{4A5i}),  $R^\a_+f =\Lam_\a T^\a_+f$, whence (use
  (\ref{tae1}) with $\nu=0$ and $p$ replaced by  $q_\a=(n+1)/(1- \a)$)
\[
\intl_{Z_n}  |(R_{+}^\a f)(\theta, t)|^{q_\a} dt d\theta= 2 \intl_{\bbr^{n}}  |(T_+^\a f) (x)|^{q_a} dx.\]
This gives the endpoint estimate in (\ref{teina}) for $p=p_\a$ and $q=q_\a$  because $T_+^\a f= \T_+^\a f$ and $\T_+^\a$ is bounded
from $L^{p_\a} (\rn)$ to $L^{q_\a} (\rn)$; see Theorems \ref {o9iu4} and  \ref{lanseT}.
 Combining the endpoint estimate with (i) and making use of the Marcinkiewicz interpolation theorem, we obtain (\ref{teina}) for all $1< p \le p_\a$,  $1/q=1-\a-n/p'$.

\vskip 0.2 truecm

\noindent { STEP III.}  Let us prove the ``only if'' part of  {\rm (ii)}.
 The  necessity of the bounds $p=p_\a$ and  $q=q_\a$ can be proved as in Theorem \ref {OST2}. Specifically,
 let $f_\lam (x)=f(\lam x)$, $\lam >0$. Then $(R^\a_+ f_\lam) (\theta, t)= \lam^{1-n-\a} (R^\a_+ f) (\theta, \lam t)$,
\[||f_\lam||_p= \lam^{-n/p} ||f||_p,\qquad || R^\a_+f_\lam ||^{\sim}_q= \lam^{1-n-\a-1/q} || R^\a_+f||^{\sim}_q.\]
 If $|| R^\a_+f||^{\sim}_q\le c\, ||f||_p$ for all $f\in L^p (\rn)$, then, replacing $f$ by $f_\lam$, we obtain
\[\lam^{1-n-\a-1/q} || R^\a_+f||^{\sim}_q \le \lam^{-n/p} ||f||_p,\]
which gives
\be\label {ordwm} 1-n-\a-1/q =-n/p \quad \text {or}\quad  1/q = 1-\a-n/p'.\ee
To obtain another relation between $p$ and $q$, we set  $\tilde f_\lam (x)= f(\lam x', x_n)$. Then  $||\tilde f_\lam||_p= \lam^{(1-n)/p} ||f||_p$, and (\ref{tae1}) (with  $p$ replaced by  $q$) yields
\[(|| R^\a_+\tilde f_\lam  ||^{\sim}_q)^q= 2 \intl_{\bbr^{n}} \!\!\frac {|(T^\a_+\tilde f_\lam) (x)|^q}{(1\!+\!|x'|^2)^{(n+(\a -1)q  +1)/2}}\, dx.\]
Note that $(T^\a_+\tilde f_\lam) (x)=\lam^{1-n} (T^\a_+ f) (x'/\lam, x_n)$.
Hence, changing variables, we obtain
\[
(|| R^\a_+\tilde f_\lam  ||^{\sim}_q)^q \!= \!\lam^{(n-1)(1-q)} A(\lam), \quad A(\lam)\!=\!2  \!\intl_{\bbr^{n}} \!\!\frac{ |(T^\a_+f) (x', x_n)|^q}{(1\!+\!|\lam x'|^2)^{(n+(\a -1)q  +1)/2}} dx.\]
If $|| R^\a_+ f||^{\sim}_q\le c\, ||f||_p$ for all $f\in L^p (\rn)$, then $|| R^\a_+ \tilde f_\lam||^{\sim}_q\le c\, ||\tilde f_\lam||_p$ for all $\lam >0$, and  we have
$\lam^{(n-1)(1-q)/q} \, A^{1/q}(\lam) \le \lam^{(1-n)/p} ||f||_p$. This gives
\be\label{urt} \lam^{s} A^{1/q}(\lam) \le  ||f||_p, \qquad s=(n-1)(1-q)/q+ (n-1)/p.\ee
 Let us pass to the limit in (\ref{urt}) as   $\lam \to 0$ assuming   $f$ to be  good enough.
If $s <0$, then the left-hand side of (\ref{urt}) tends to infinity, which gives contradiction. Hence  $s \ge 0$, that is, $q\le p'$.
Combining the last inequality with (\ref{ordwm}), we obtain $p\le p_\a$, as desired.

To complete the proof, it remains to show that  (\ref{teina}) fails  if $p=1$. This can be done using  the approximation argument; cf. \cite[p. 119]{Ste},
where  similar reasoning   was applied to the Riesz potentials in $\rn$. Suppose the contrary, that is, there is a constant $c$, such that
\be\label {teinas}
\|R^\a_+f\|^{\sim}_{1/(1-\a)} \le c\,  \| f \|_1  \quad \text {for all}\quad f\in L^1(\rn). \ee
We choose $f$ to be a mollifier
\be\label {omee} \om_\e (x)=\left \{ \begin{array} {ll} \displaystyle{\frac{C}{\e^n}\, \exp\left (-\frac{\e^2}{\e^2-|x|^2}\right )},& |x|\le \e,\\
0,& |x|> \e,\\\end{array}\right .\ee
where $C$ is chosen so that $\int_{\bbr^{n}}\om_\e (x)\, dx=1$. Then
\be\label {omf} (R_+^{\,\a}\om_\e)(\theta, t) = \frac{1}{\Gam (\a)}\intl_{\bbr} (t-\eta)_+^{\a -1} (R\om_\e)(\theta, \eta)\, d\eta.\ee
Clearly,  $\om_\e (x)= \e^{-n} \om_1 (x/\e)$ and $(R\om_\e)(\theta, \eta) =\e^{-1} (R\om_1)(\theta, \eta/\e)$. We  set
 $\om_1 (x)=\om (|x|)$. Then for all $\theta \in S^{n-1}$,
\[ (R\om_1)(\theta, \eta)= \psi (\eta), \quad
\psi (\eta)=\sigma_{n-2} \intl^\infty_{|\eta|}\! \om (r)(r^2-\eta^2)^{(n-3)/2}r dr; \]
see (\ref{rese}). Thus (\ref{omf}) becomes
\be\label {tetas}
(R_+^{\,\a}\om_\e)(\theta, t) = \frac{t_+^{\a -1}}{\Gam (\a)} \ast \psi_\e, \qquad \psi_\e (\eta)= \e^{-1} \psi (\eta/\e).\ee
The family $\{\psi_\e\}_{\e >0}$ is an approximate identity. To see that, we observe that by (\ref{vali}),
\[
||\psi||_{L^1 (\bbr)}=\intl_{\bbr} (R\om_1)(\theta, \eta) d\eta = ||\om_1||_{L^1 (\rn)}=1.\]
Further,  $\psi$ is a monotone decreasing function of $|\eta|$. The latter is obvious for $n=3$ and can be easily checked for $n>3$, when  differentiation yields $\psi'(\eta)<0$ for $\eta >0$.  In the case $n=2$, $\eta >0$, we first integrate by parts to get
$\psi (\eta)=-2 \int^\infty_{\eta}\! \om' (r)(r^2-\eta^2)^{1/2} dr$ and then differentiate, which gives $\psi' (\eta)=2 \eta\int^\infty_{\eta}\! \om' (r)(r^2-\eta^2)^{-1/2} dr<0$.

Now applying  the approximation to the identity machinery  \cite[Theorem 2(b), p. 63]{Ste} to the function
\[
t_+^{\a -1}= t_+^{\a -1} \chi {(0,1)} + t_+^{\a -1} \chi {(1,\infty)} \in L^1 (\bbr) +  L^q (\bbr), \qquad q>1/(1-\a),\]
we conclude that $(R_+^{\,\a}\om_\e)(\theta, t)$ converges to  $ t_+^{\a -1}/\Gam (\a)$ a.e. on $\bbr$, and hence for almost all $(\theta, t) \in Z_n$. By  Fatou's lemma and the assumption (\ref{teinas}) it follows that
\bea
\intl_{\bbr} |t_+^{\a -1}/\Gam (\a)|^{1/(1-\a)} dt &=& \sig_{n-1}^{-1}\,||  \lim\limits_{\e \to 0} \,(R_+^{\,\a}\om_\e)(\theta, t)||^{\sim}_{1/(1-\a)} \nonumber\\
&\le&  \sig_{n-1}^{-1}\, \underset{\e \to 0}{\lim\limits \inf} \,|| R_+^{\,\a}\om_\e||^{\sim}_{1/(1-\a)} \le c \,||\om_\e||_1=c,\nonumber\eea
which is impossible, because the integral on the left-hand side diverges.
\end{proof}

\begin{remark}\label {8uy} As we pointed out in Introduction, it is an interesting open problem to obtain necessary and sufficient conditions of the $L^p$-$L^q$ boundedness of the extended operator $\R_+^\a$ for {\it all} complex $\a$. The above theorems contain such conditions only for $\a=0$ and  $0< Re \, \a < 1$. As we could see, the theory of the Radon-type fractional integrals  $R_+^\a f$,
associated with the classical Radon transform $R$ is more complicated than that for
$T_+^\a f$, associated with the transversal Radon transform. One of the explanations of this phenomenon might be  that $R_+^\a $ can be scaled only in the radial direction, whereas $T_+^\a $ enjoys
 bi-parametric dilations (in the $x'$-variable and in the $x_n$-variable).
\end{remark}

\section{Parabolic Radon-type  fractional integrals }

In this section we consider the parabolic Radon transform
\be \label {pparP}
(P f)(x) =\!\intl_{\bbr^{n-1}} f(x'\!-\!y', x_n \! -\!|y'|^2)\, dy', \quad  x=(x', x_n) \in \rn, \ee
 and  the corresponding  fractional integral
\be\label{poaxzc}
(P_{+}^{\,\a} f)(x) =\frac{1}{\Gam (\a)} \intl_{\bbr^n}(y_n -|y'|^2)_{+} ^{\a -1} \,f(x-y)\, dy.\ee
   The  integral (\ref{pparP})  resembles integration  over the shifted  paraboloid
\[\pi_x= \pi_0 +x, \qquad \pi_0= \{y=(y', y_n):  y_n= - |y'|^2\}.\]
However, $(P f)(x)$ differs from the usual surface integral
\be \label {papar}
\intl_{\pi_x} \!f(y)\, d\sig (y)=\intl_{\bbr^{n-1}} f(x'-y', x_n  -|y'|^2) \, (1+4|y'|^2)^{1/2} dy'\ee
 by the Jacobian factor, which is suppressed in our consideration.

Our aim is to obtain sharp $L^p$-$L^q$ estimates for $P f$ and  $P_{+}^{\,\a} f$.

\subsection{The parabolic Radon transform}
If $f\in L^1 (\rn)$, then, by Fubini's theorem,
\be\label{valip}
\intl_{-\infty}^\infty (Pf) (x', x_n) dx_n=\intl_{\bbr^n} f(x) dx,\ee
whence  $(Pf) (x', x_n)$ is finite for all
$x' \in \bbr^{n-1}$ and almost all $x_n\in \bbr$.

There is a remarkable connection between $P$ and the transversal Radon transform $T$ in (\ref {bart}).
This connection was pointed out  in \cite [Lemma 2.3]{Chr} and  used in \cite {Ru22}. For the sake of completeness, below
we describe it in  detail. Let
\be \label {bars}
(B_1f)(x)\!=\!f(x', x_n \!-\!|x'|^2), \quad (B_2 F)(x)\!=\!F(2x',  x_n \!-\!|x'|^2).\ee
The corresponding inverse maps have the form
\be \label {bars1} (B_1^{-1}u)(x) \! = \!u(x', x_n \!+\!|x'|^2), \quad (B_2^{-1} v)(x)\!=\!v \left (\frac{x'}{2},  x_n \!+\!\frac{|x'|^2}{4}\right ).\ee
One can readily see that
\be\label {iar} ||B_1f||_p =||f||_p, \qquad \| B_2 F \|_{q} = 2^{(1-n)/q} \,\|F \|_{q}.\ee

\begin{lemma}\label {swa} The equality
\be \label {bts}  Pf=B_2TB_1 f,\ee
holds provided that either side of it exists in the Lebesgue sense.
\end{lemma}
\begin{proof} We write the left-hand side as
\bea
(P f)(x) &=& \intl_{\bbr^{n-1}} f(y', x_n-|x'-y'|^2)\, dy'\nonumber\\
&=& \intl_{\bbr^{n-1}} f(y', x_n -|x'|^2 -|y'|^2 + 2 x'\cdot y')\, dy'.\nonumber\eea
Hence
\[
(B_2^{-1}P f)(x)= (P f)\!\left (\frac{x'}{2},  x_n +\frac{|x'|^2}{4}\right )=\intl_{\bbr^{n-1}}\! \!  f (y', x_n -|y'|^2+ x'\cdot y')\, dy'.\]
On the other hand,
\[
(TB_1f)(x)=  \intl_{\bbr^{n-1}} \!\!  (B_1f)(y', x'\cdot y' +x_n)\, dy'=\intl_{\bbr^{n-1}}\! \!  f (y', x_n -|y'|^2+ x'\cdot y')\, dy',\]
as above. This gives the result.
\end{proof}

\begin{lemma}\label{swq} Let $v(x)\!=\!(1+ |x'|^2 +(x_n + |x'|^2)^2)^{-1/2}$. Then
  \be \label{eq2z}\intl_{\bbr^n}  (Pf)(x) u(x) dx
=\frac{\sig_{n-1}}{2^n}\, \intl_{\rn} f(x) v(x) dx, \ee
provided that either side of this  equality exists in the Lebesgue sense.
\end{lemma}
\begin{proof}
 The formula (\ref{eq2z}) follows from (\ref{eq2}) and (\ref{bts}).
\end{proof}

\begin{lemma} \label{lseedP} If $1 \le p < n/(n-1)$, then $(Pf)(x)$ is finite for almost all  $x\in \rn$ and the absolute value of the
left-hand side of (\ref{eq2z}) does not exceed $ c\, ||f||_p$, $\, c=\const$.
\end{lemma}

This statement follows from (\ref{eq2z})  by H\"older's inequality.

\begin{remark} The bound $p < n/(n-1)$ in Lemma \ref{lseedP} is sharp. Indeed, suppose $p \ge n/(n-1)$ and let $f_0 (\in
L^p(\bbr^n))$ be a function for which $ T f_0\equiv \infty$; see Remark \ref{kaz}. The function
 $f_{*}= B_1^{-1} f_0$  belongs to $L^p(\bbr^n)$ by (\ref{iar}).
By (\ref{bts}),  it follows that $Pf_{*}= B_2TB_1 f_{*}\equiv \infty$.
\end{remark}

\subsection {Fractional integrals $P_{+}^{\,\a}f$. Elementary properties}

The following lemma is a generalization of (\ref{bts}).

\begin{lemma}\label {swaw} Let $B_1$ and $B_2$ be the mappings (\ref{bars}). If  $Re \,\a >0$, then
\be \label {btswh}  P_{+}^{\,\a} f\!=\!B_2 T_{+}^\a B_1 f, \ee
 provided that either side  of this equality exists in the Lebesgue sense.
\end{lemma}
\begin{proof}  Owing to (\ref{bars1}),
\bea
&&(B_2^{-1}P_{+}^{\,\a} f)(x)\!= \!\frac{1}{\Gam (\a)} \intl_{\bbr^n}(y_n \!-\!|y'|^2)_{+} ^{\a -1} \,f \left( \frac{x'}{2}\!- \!y', x_n \!+\! \frac{|x'|^2}{4} \! -\! y_n \right) dy,\nonumber\\
&& =\frac{1}{\Gam (\a)}  \intl_{\bbr^{n-1}} dy' \intl_0^\infty s^{\a -1} f \left( \frac{x'}{2}- y', x_n + \frac{|x'|^2}{4}  - s- |y'|^2\right ) ds\nonumber\\
\qquad \label {mmnzh} && =\frac{1}{\Gam (\a)}  \intl_0^\infty s^{\a -1} ds  \intl_{\bbr^{n-1}}  f (z', x_n -s + x'\cdot z'- |z'|^2)\, dz'.\eea
On the other hand,
\bea
&&(T_{+}^\a B_1f)(x)=\frac{1}{\Gam (\a)}\intl_{\rn} (x_n -y_n)_{+}^{\a -1} (B_1f)(y', y_n + x' \cdot y')\, dy\nonumber\\
&&= \frac{1}{\Gam (\a)}  \intl_0^\infty s^{\a -1} ds \intl_{\bbr^{n-1}}  f (y',  x_n -s + x'\cdot y'- |y'|^2)\, dy',\nonumber\eea
which coincides with (\ref{mmnzh}).
\end{proof}

Theorem \ref {o9iu2} gives the following boundedness result for  $P_{+}^\a f$.

 \begin{lemma}\label {o9iu3}  Let  $\,0<\a <1$,
\be\label{op98ypa} 1\le p < \frac{n}{n\!-\!1\!+\!\a}, \quad \nu= -\a-\frac {n\!-\!1}{p'}, \quad \mu = -n \left (1\!-\!\frac{2}{p}\right),\ee
$1/p +1/p' =1$. We define
\[ L^p_v(\bbr^n)\!=\!\{ f: ||f||_{p,v}\equiv || v(x) f||_{L^p(\bbr^n)}<\infty \}, \quad v(x) \!=\! \frac {|x_n \!-\! |x'|^2|^\nu}{(1\!+\! 4|x'|^2)^{\mu/2}}.\]
Then
\be\label{op982y}
||P_{+}^\a f||_{p,v}\le  c \,||f||_{p},\qquad c=\const.\ee
\end{lemma}
 \begin{proof}  Let $u(x) = |x_n|^\nu (1+|x'|^2)^{-\mu/2}$, as in Theorem \ref {o9iu2}.
By (\ref{btswh}) and (\ref {bars1}),
\bea
||T_+^\a f||^p_{p,u}&=&||u(x) B_2^{-1}P_{+}^{\,\a}  B_1^{-1} f||_p^p\nonumber\\
&=&\intl_{\bbr^{n}}\frac{|x_n|^{\nu p}}{(1\!+\!|x'|^2)^{\mu p/2}} \left | (P_{+}^{\,\a}  B_1^{-1} f)\left (\frac{x'}{2},  x_n +\frac{|x'|^2}{4}\right )\right |^p dx. \nonumber\eea
Changing variables, we  write the last expression as
\[
\intl_{\bbr^{n}} \left| (P_{+}^{\,\a}  B_1^{-1} f)(y) \frac {|x_n \!-\! |x'|^2|^\nu}{(1\!+\! 4|x'|^2)^{\mu/2}}\right |^p dy = ||P_{+}^{\,\a}  B_1^{-1} f||^p_{p,v}.\]
Hence, setting $\vp =B_1^{-1} f$, we obtain
\[
 ||P_{+}^{\,\a} \vp||^p_{p,v}= ||T_+^\a B_1 \vp||^p_{p,u} \le 2^{-1/p} c_\a \,||B_1 \vp||_{p}= 2^{-1/p} c_\a \,||\vp||_{p}; \]
cf. (\ref{iar}). This gives (\ref{op982y}), up to notation.
\end{proof}

\begin{lemma} \label {side} Let  $f \!\in S (\rn)$.  The following statements hold.

\noindent {\rm (i)} For each $x \in \rn$, $(P_{+}^{\,\a} f)(x)$
 extends as an entire function of $\a$. Moreover,
\be\label {Dac} \lim\limits_{\a \to 0} (P_{+}^{\,\a} f)(x) =(Pf)(x),\ee
where $(P f)(x)$ is the parabolic Radon transform (\ref{pparP}).

\noindent {\rm (ii)}  If $Re \,\a  > 0$, then  for any multi-index $m$,
\be\label {Dacw} \partial^m P_{+}^{\,\a} f =P_{+}^{\,\a} \partial^m f.\ee

\noindent {\rm (iii)} For any positive integer $k$,
\be\label {Dacwr}
 P_{+}^{\,\a} f =  P_{+}^{\,\a +k} \partial_n^k  f=  \partial_n^k P_{+}^{\,\a +k} f. \ee
\end{lemma}
\begin{proof} To prove {\rm (i)}, we have
 \bea
(P_{+}^{\,\a} f)(x)\!\!&=&\!\!\frac{1}{\Gam (\a)}\intl_{\bbr^{n-1}} dy'\intl_{-\infty}^\infty  (y_n -|y'|^2)_{+}^{\a -1}  f (x' -y', x_n -y_n)\, dy_n \quad \nonumber\\
\label {008}&=&\!\!\frac{1}{\Gam (\a)}\intl_0^\infty \!s ^{\a -1} A_{x} (s)\,ds; \\
A_{x} (s)\!\!&=&\!\!\intl_{\bbr^{n-1}} \!\!\! f (x' -y', x_n  - |y'|^2 - s)\, dy'.\nonumber\eea
Because the function $A_{x} (s)$ is smooth, rapidly decreasing, and satisfies $A_{x} (0)=(Pf)(x)$, the result follows; cf. \cite [Chapter I, Section 3.2]{GS1}, \cite [Section 2.5]{Ru15}.
 In  {\rm (ii)} we simply differentiate under the sign of integration.  The first equality in (\ref {Dacwr})  can be obtained using integration by parts. The second equality is the result of differentiation:
 \[
  \partial_n^k P_{+}^{\,\a +k} f= \partial_n^k I_{+}^{k}P_{+}^{\,\a} f=P_{+}^{\,\a} f.\]
\end{proof}

\subsection{The main theorem for $P_{+}^{\,\a} f$  }

\begin {theorem} \label {lanseTp} Suppose $1\le p,q \le \infty$, $\a_0=Re\, \a$.

\vskip 0.2 truecm

\noindent {\rm (i)}
The operator $P_{+}^{\,\a}$   initially defined on functions $\vp \in S (\rn)$ by analytic continuation, extends as a linear bounded operator  $\P_{+}^{\,\a}:  L^p (\rn)\to L^q (\rn)$  if and only if
\[
\frac{1-n}{2} \le \a_0 \le 1, \qquad p=\frac{n+1}{n + \a_0}, \qquad q=\frac{n+1}{1-\a_0}.\]
 In particular, the parabolic Radon transform $P$ extends  as a linear bounded operator  $\P: L^p (\rn)\to L^q (\rn)$ if and only if
$p=(n+1)/n$ and  $ q=n+1$.

\vskip 0.2 truecm

\noindent {\rm (ii)} In the cases $0<Re\, \a <1$ and  $\a=0$, the $L^q$-functions $(\P_{+}^{\,\a} f)(x)$ and $(\P f)(x)$
coincide a.e.  with   absolutely convergent integrals  $(P_{+}^{\,\a} f)(x)$ and $(P f)(x)$, respectively.
\end{theorem}
\begin{proof}  Let $f\in S(\rn)$. If  $Re \,\a >0$, then by (\ref{btswh}),  $P_{+}^{\,\a} f\!=\!B_2 T_{+}^\a B_1 f$.  By  Lemmas \ref{sideT} and
\ref{side}, this equality extends analytically to all $\a \in \bbc$, and the  analytic
 continuation represents a smooth function on  $\rn$. Here we take into account  that $B_1$ and $B_2$ are diffeomorphisms of $ S(\rn)$.  We keep the notation $P_{+}^{\,\a}$ and  $T_{+}^{\,\a}$ for the corresponding analytic continuations. Then, by  (\ref{iar}) and Theorem \ref{lanseT},
\bea
|| P_+^\a f||_{q} &=& || B_2 T_{+}^\a B_1 f||_{q}=  2^{(1-n)/q} \,\|T_{+}^\a B_1 f \|_{q}    \nonumber\\
&\le& c\,||B_1 f||_p =c\, ||f||_p,\nonumber\eea
as desired. The proof of {\rm (ii)}  mimics the reasoning from Theorem \ref{o9iu4}.
\end{proof}

\section{Explicit formulas for  operators defined by interpolation }

 We restrict our consideration to the operators $T_{+}^{\,\a}$. Similar reasoning is applicable to the operators $R_{+}^{\,\a}$, $P_{+}^{\,\a}$, and their modifications. The cases $Re \a >0$ and $\a=0$, when the extended operators  $\R_{+}^{\,\a}$ and  $\P_{+}^{\,\a}$ coincide a.e. with the corresponding absolutely convergent integrals were mentioned in Theorems  \ref{o9iu4}  and \ref{lanseTp}.
  The most intriguing are the remaining cases, when these integrals diverge.

  We invoke the regularization technique, which amounts to the concept of {\it Marchaud's fractional derivative} \cite{Marc}. The latter is well known in Fractional Calculus and has proved to be useful for regularization of divergent integrals with power singularity; see, e.g., \cite{Ru89a, Ru89, Ru96, Ru15, Sam, SKM}.
 For the sake of simplicity, we restrict to the case of real $\a \in [(1-n)/2, 0)$. The same idea can be applied to the case $Im\, \a \neq 0$, but  the formulas look a bit  more complicated; cf. \cite{Marc, Ru96}.

 Let us recall some known facts.
The  Marchaud fractional derivative $\bbd^\a_{+}\vp$ of order $\a> 0$ of a function $\vp : \bbr \to \bbc$ is defined by
\be \label{cv56jks0}
(\bbd^\a_{+}\vp)(x)=\frac{1}{\varkappa_\ell (\a)} \intl
_0^\infty (\Del^\ell_{t}\vp) (x)
\,\frac{dt}{t^{1+\a}}, \qquad x\in \bbr.\ee
Here $(\Del^\ell_{t}\vp) (x)$ is the finite difference
\be (\Del^\ell_{t}\vp) (x)=
\sum_{j=0}^\ell {\ell \choose j} (-1)^j \vp (x - jt),\qquad \ell >\a, \ee
and the normalizing constant $\varkappa_\ell (\a)$ has the form
\bea\label{gruks1}
\qquad \varkappa_\ell (\a)&\equiv& \intl_{0}^{\infty }\frac{\left( 1-e^{-v}\right)^{\ell}}{v^{\alpha
+1}}\,dv\\
&=&\left\{ \!
 \begin{array} {ll} \Gam (-\a) \displaystyle{\sum\limits_{j=1}^\ell {\ell \choose j} (-1)^j
j^\a}, & \mbox{ $ \a \neq 1,2, \ldots , \ell -1,$}\\
\displaystyle{ \frac{(-1)^{1+\a}}{\a !}  \sum\limits_{j=1}^\ell
{\ell \choose j} (-1)^j j^\a
\log j}, & \mbox{$\a= 1,2, \ldots , \ell -1. $}\\
  \end{array}
\right. \nonumber \eea

If $\a =m$ is a positive integer and $\vp$ is good enough, then $(\bbd^\a_{+}\vp)(x)=\vp^{(m)} (x)$ is the usual $m$th derivative of $\vp$.
One can show that
\be\label {mmnnz}
\bbd^\a_{+} I^\a_{+} f= \lim\limits_{\e \to 0}\bbd^\a_{+, \e}  I^\a_{+} f =f,\ee
where
\be\label{trhsi} (\bbd^\a_{+, \e}\vp)(x)=\frac{1}{\varkappa_\ell (\a)} \intl
_\e^\infty (\Del^\ell_{t}\vp) (x)
\,\frac{dt}{t^{1+\a}}, \qquad \e>0,\ee
is the  {\it truncated Marchaud fractional derivative}.
The proof of  (\ref {mmnnz}) relies on the representation of $\bbd^\a_{+, \e}  I^\a_{+} f$ as an approximate identity
\be\label{trhsika}(\bbd^\a_{+, \e}  I^\a_+ f)(x)=\intl_0^\infty \lam_{\ell,\a} (\eta)\, f(x - \e \eta)\, d\eta,\ee
where the averaging kernel  $\lam_{\ell,\a} (\eta)$ is defined by
\be\label{gruks11}
\lam_{\ell, \a} (\eta)=\frac{1}{\eta\, \Gam (1+\a) \,\varkappa_\ell (\a)}\,
\sum_{j=0}^\ell {\ell \choose j} (-1)^j (\eta-j)_{+}^{\a}, \ee
and has the following properties
\be\label{stulka} \intl_0^\infty \lam_{\ell, \a} (\eta) d\eta =1, \qquad
\lam_{\ell, \a} (\eta)=\left\{ \!
 \begin{array} {ll} O(\eta^{\a -1}) & \mbox{if $ \eta <1,$}\\
 O(\eta^{\a -\ell-1}) & \mbox{if $\eta >1; $}\\
  \end{array}
\right. \ee
see, e.g.,  \cite [p. 182]{Ru88} \cite [pp. 50, 51]{Ru15}.  If $f\in L^p (\bbr)$, $p \in [1,\infty)$, then the expression in (\ref{trhsika}) tends to $f$  as $\e \to 0$
 in the $L^p$-norm and in
the almost everywhere sense. If, moreover, $f\in C_0(\bbr)$, the limit exists in the sup-norm. The result is independent of the choice of the integer $\ell >\a$.

Let us proceed to regularization of  $\T^\a_+ f$.
 Below we  use the above results for Marchaud's fractional derivative with $\a$  replaced by $-\a$ and  obtain explicit formulas for  $\T^{\a}_{+} f$ in the case $(1-n)/2 \le \a <0$.
 Suppose  $f\in L^p (\rn)$,  $ p=(n+1)/(n + \a)$, and  set
\[\vp = \T f \, \qquad \psi = \T_{+}^{\,\a}f. \]
Recall that  by Theorem \ref{lanseT},
\[\vp \in L^r (\rn), \; r\!=\!(n\!+\!1)/n,\; \text{\rm and}\;  \psi\in L^q (\rn), \; q\!=\! (n\!+\!1)/(1\!-\!\a).\]
Moreover, by Theorem \ref{o9iu4},
$\vp = \T f$ can be written in the integral form
\[
\vp (x)=(T f)(x) =\intl_{\bbr^{n-1}}  f(y', x_n + x'\cdot y')\, dy'. \]

\begin{theorem} \label {iffe} Let $f\in L^p (\rn)$,
\[  p=(n+1)/(n + \a), \qquad q= (n+1)/(1-\a), \qquad (1-n)/2 \le \a <0.\]
 Then the function  $\T^{\a}_{+} f\in L^q (\rn)$,
 determined by Theorem \ref{lanseT}, can be represented by the difference hypersingular integral
\be\label {easo}
(\T^{\a}_{+} f)(x', x_n)\!=\!\frac{1}{\varkappa_\ell (-\a)} \!\intl_0^\infty \!\left [\sum_{j=0}^\ell {\ell \choose j} (-1)^j (T f)(x', x_n \!-\! jt)\right ]\frac{dt}{t^{1-\a}},\ee
in which $\varkappa_\ell (-\a)$ is defined by (\ref{gruks1}),  $\ell > -\a$, and  $\int_0^\infty (...)=\lim\limits_{\e \to 0}\int_\e^\infty (...)$. The limit exists
 in the  $L^q$-norm with respect to the $x_n$-variable  for almost all $x' \in \bbr^{n-1}$. It also exists for almost all $x\in \rn$.
\end{theorem}

\begin{example} Let $n=2$, when
\[p=3/(2 + \a), \qquad q= 3/(1-\a), \qquad -1/2 \le \a <0.\] Choosing $\ell=1$, we obtain
\[
(\T^{\a}_{+} f)(x', x_n)\!=\!\frac{1}{\Gam (\a)} \intl_0^\infty \!\frac{(T f)(x', x_n - t)-(T f)(x', x_n)}{t^{1-\a}} \, dt.\]
\end{example}

 \noindent{\it Proof of Theorem \ref{iffe}.}
 Recall that  $\vp = \T f =Tf \in L^r (\rn)$, $\psi = \T_{+}^{\,\a}f\in L^q (\rn)$.
 We consider $\vp$ and $\psi$ as single-variable functions
 $\vp_{x'} (x_n)\equiv\vp  (x', x_n)$ and   $\psi_{x'} (x_n)\equiv \psi (x', x_n)$.  Clearly, $\vp_{x'} (\cdot)\in L^r (\bbr)$ and $\psi_{x'} (\cdot)\in L^q (\bbr)$  for almost all $x'\in \bbr^{n-1}$.
Denote
\[
 (\Lam_\e \psi_{x'})(x_n)= \intl_0^\infty \lam_{\ell,-\a} (\eta) \psi_{x'} (x_n - \e \eta) d\eta, \qquad \e >0.\]
  Our nearest aim is to show that
\be\label{grutg}
(\bbd^{-\a}_{+, \e}\vp_{x'})(x_n)=(\Lam_\e \psi_{x'})(x_n), \ee
where $\bbd^{-\a}_{+, \e}$ stands for the truncated Marchaud fractional derivative (\ref{trhsi}) of order $-\a$ and $\lam_{\ell,-\a} (\eta)$ is the averaging kernel  (\ref{gruks11}) with $\a$ replaced by $-\a$.
 The next step will be passage to the limit in (\ref{grutg}) as $\e \to 0$.

Let $\Phi_0 (\bbr)$ be the Semyanistyi-Lizorkin space of test functions on the real line; see   Definition  \ref{Semyanistyi}. Because $\bbd^{-\a}_{+, \e}\vp_{x'} \in  L^r (\bbr)$ and
$\Lam_\e \psi_{x'} \in  L^q (\bbr)$ for almost all $x'\in \bbr^{n-1}$, then, by Proposition \ref{iqah}, it is enough to prove (\ref{grutg}) in a weak sense, i.e.,
\be\label{grutgz}
\langle \bbd^{-\a}_{+, \e}\vp_{x'}, \om \rangle=\langle\Lam_\e \psi_{x'}, \om \rangle, \qquad \om \in \Phi_0 (\bbr).\ee
 We have
\be\label {rchaud}
\langle\bbd^{-\a}_{+, \e}\vp_{x'}, \om \rangle= \langle\bbd^{-\a}_{+, \e} [T f (x', \cdot)], \om \rangle= \langle (T f) (x', \cdot), \tilde\om \rangle,\ee
where $\tilde \om= \bbd^{-\a}_{-, \e}\om \in \Phi_0 (\bbr)$.

By  Lemma \ref{izsh},  there is a sequence $\{f_k\} \subset \Phi (\rn)$, converging to $f$ in the $L^p$-norm.
Then, by Theorem \ref{lanseT} and the equality $\T f=Tf$, the sequence    $\{T f_k\}$ converges to $Tf$ in  $L^r (\rn)$, that is,
\[
||Tf - T f_k||_r^r= \!\!\intl_{\bbr^{n-1}}\!\!\! dx'\!\! \intl_{\bbr} \!\!|(Tf)(x', x_n)-  (Tf_k)(x', x_n)|^r dx_n \to 0, \quad k \to \infty.\]
It follows that there is a subsequence $\{k_j\}$, such that
\[\intl_{\bbr} \!\!|(Tf)(x', x_n)- (Tf_{k_j})(x', x_n)|^r dx_n \to 0, \quad \text {\rm as} \quad j\to \infty,\]
 for almost all $x' \in \bbr^{n-1}$. Hence, for almost all $x' \in \bbr^{n-1}$, by H\"older's inequality we obtain
\bea
&& |\langle (T f) (x', \cdot), \tilde\om \rangle- \langle (Tf_{k_j}) (x', \cdot), \tilde\om \rangle|\nonumber\\
 &&\le \Big (\intl_{\bbr} \!\!|(Tf)(x', x_n)- (Tf_{k_j}(x', x_n)|^r dx_n\Big )^{1/r} ||\tilde\om||_{r'}  \to 0, \quad j\to \infty,\nonumber\eea
and (\ref{rchaud}) can be continued:
\bea
\langle\bbd^{-\a}_{+, \e}\vp_{x'}, \om \rangle&=& \lim\limits_{j\to \infty} \langle(T f_{k_j}) (x', \cdot), \tilde\om \rangle\nonumber\\
&=& \lim\limits_{j\to \infty} \langle I^{-\a}_{+} I^{\a}_{+}[(T f_{k_j}) (x', \cdot)], \tilde\om \rangle\nonumber\\
&=& \lim\limits_{j\to \infty} \langle I^{\a}_{+}[(T f_{k_j}) (x', \cdot)],  I^{-\a}_{-} \bbd^{-\a}_{-, \e}\om \rangle\nonumber\\
&=& \lim\limits_{j\to \infty} \langle(T^{\a}_{+} f_{k_j}) (x', \cdot),  \bbd^{-\a}_{-, \e} I^{-\a}_{-} \om \rangle,\nonumber\eea
where
\[ (\bbd^{-\a}_{-, \e}\vp)(x)=\frac{1}{\varkappa_\ell (-\a)} \intl
_\e^\infty  \Big [\sum_{j=0}^\ell {\ell \choose j} (-1)^j \vp (x + jt)\Big ]
\,\frac{dt}{t^{1-\a}}.\]
Note that $\bbd^{-\a}_{-, \e}$ preserves the space $\Phi_0 (\bbr)$ and the operators  $I^{-\a}_{-}$ and $\bbd^{-\a}_{-, \e}$ commute in this space. Note also that
$I^{\a}_{+}T f_{k_j}=  T^{\a}_{+} f_{k_j}$ because $f_{k_j} \in  \Phi (\rn)$.

Furthermore, since   $\{f_{k_j}\}$,  as a subsequence of  $\{f_{k}\}$, converges to $f$ in the $L^p$-norm, by Theorem \ref{lanseT} it follows that $ T^{\a}_{+} f_{k_j} \to  \T_{+}^{\,\a}f$ in the $L^q$-norm. Hence the last limit can be written as $\langle(\T^{\a}_{+} f) (x', \cdot),  \bbd^{-\a}_{-, \e} I^{-\a}_{-} \om \rangle$ or
$\langle\psi_{x'},   \bbd^{-\a}_{-, \e} I^{-\a}_{-} \om \rangle$. The composition $\bbd^{-\a}_{-, \e} I^{-\a}_{-} \om$ can be transformed by the formula
\[(\bbd^{-\a}_{-, \e}  I^{-\a}_- f)(x)=\intl_0^\infty \lam_{\ell,-\a} (\eta)\, f(x + \e \eta)\, d\eta,\]
which is a modification of (\ref{trhsika}).
Hence, we continue:
\bea
\langle\bbd^{-\a}_{+, \e}\vp_{x'}, \om \rangle&=&\intl_{\bbr}\psi_{x'} (x_n), dx_n \intl_0^\infty \lam_{\ell,-\a} (\eta)\, \om_{x'}(x_n + \e \eta)\, d\eta \nonumber\\
&=&\intl_0^\infty \lam_{\ell,-\a} (\eta) \, d\eta \intl_{\bbr}\psi (x', x_n)\,\om_{x'}(x_n + \e \eta)\,dx_n \nonumber\\
&=&\intl_0^\infty \lam_{\ell,-\a} (\eta) \, d\eta \intl_{\bbr} \psi (x',  x_n - \e \eta) \,\om_{x'}(x_n)\,  dx_n \nonumber\\
&=&\intl_{\bbr} (\Lam_\e \psi_{x'})(x_n) \,  \om_{x'}(x_n) dx_n.\nonumber\eea
This gives (\ref{grutgz}), and therefore (\ref{grutg}).

To complete the proof, it remains to apply the standard machinary of approximation to the identity   to the right-hand side of (\ref{grutg}) in the $x_n$-variable, assuming $x'$ fixed.
Taking into account that  $\psi_{x'})(\cdot)\in L^q (\bbr)$ and using the properties (\ref{stulka}) of the averaging kernel $\lam_{\ell, -\a} (\eta)$, we conclude that for almost all $x'\in \bbr^{n-1}$,
$(\Lam_\e \psi_{x'}) (x_n)\to \psi_{x'} (x_n)$ in the norm of the space $L^q (\bbr)$ and for almost all $x_n\in \bbr$.
Because the Marchaud fractional derivative $(\bbd^{-\a}_{+}\vp_{x'})(x_n)=\lim\limits_{\e \to 0}  (\bbd^{-\a}_{+, \e}\vp_{x'})(x_n)$ coincides with the right-hand side of (\ref {easo}), we are done.
\hfill $\Box$

\section{Conclusion}

Many other Radon-type fractional integrals are known in  harmonic analysis and integral geometry.  Such integrals are
 associated with Radon transforms on constant curvature spaces, matrix spaces, and diverse homogeneous spaces of Lie groups;  see, e.g., \cite{OP, OPR, OR, Ru13, Ru13a, Str81}. The corresponding $L^p$-$L^q$ estimates  are of great interest.
 Perhaps some readers will  be inspired to develop interpolation tools that  would be applicable to  analytic families of this kind.


\begin{thebibliography}{[ASMR]}

\bibitem  {AS} M. Abramowitz and I.A. Stegun (ed.) {\it Handbook of mathematical functions with formulas, graphs, and mathematical tables},  US Department of Commerce, National Bureau of Standards, 10th Printing, December 1972.

\bibitem  {AAR}  G.E. Andrews, R. Askey, and R. Roy,  {\it Special functions},
 Cambridge University Press, 2001.

\bibitem  {BS} C. Bennett and R. Sharpley, {\it Interpolation of operators}, Academic Press, Inc., 1988.

\bibitem  {BK}  A.L. Buhgeim and V.B. Kardakov,  Solution of an inverse problem for an elastic wave equation by the method of spherical means (in Russian) {\it Sibirsk. Mat. Ž.} {\bf 19} (1978), no. 4, 749--758, 953.

%\bibitem  {Cal} A. P. Calder\'on. On the Radon transform and some of its generalizations.  In \textit{Conference on harmonic analysis in honor of Antoni Zygmund, Vol. I, II (Chicago, Ill., 1981)},  Wadsworth Math. Ser., Wadsworth, Belmont, CA, 1983, 673--689.

\bibitem  {Chr} M. Christ,  Extremizers of a Radon transform inequality,  \textit{Advances in analysis: the legacy of Elias M. Stein}, 84--107, Princeton Math. Ser., 50,  Princeton, NJ: Princeton Univ. Press, 2014.

\bibitem  {E} L. Ehrenpreis, {\it The Universality of the Radon transform}, Oxford University Press, 2003.

\bibitem {Es}  G.I. Eskin.  \textit{Boundary value problems for elliptic
 pseudodifferential equations}. Amer. Math. Soc., Providence, R.I.,  1981.

\bibitem {GS1} I.M. Gelfand and G.E. Shilov, \textit{Generalized functions, vol. 1. Properties and  operations}, Academic
Press, New York, 1964.

\bibitem {Graf04}  L. Grafakos,  \textit{Classical and modern Fourier analysis},  Pearson Education Inc., Prentice Hall, 2004.

\bibitem  {Graf} \bysame, Fundamentals of Fourier analysis,  Springer GTM series, to appear, 2023.

\bibitem  {GO} L. Grafakos and E.M. Ouhabaz, Interpolation for analytic families of multilinear
operators on metric measure spaces, {\it Studia Mathematica} {\bf 267} (2022) , 37--57.


\bibitem  {H11} S. Helgason,  \textit{Integral geometry and Radon transform}, Springer, New York-Dordrecht-Heidelberg-London, 2011.

\bibitem {H65}  \bysame,  The Radon transform on Euclidean spaces, compact two-point
homogeneous spaces and Grassmann manifolds,  \textit{Acta  Math.} \textbf{113}  (1965), 153--180.

\bibitem {Liz} P.I. Lizorkin,   Generalized Liouville differentiation and functional spaces $L_p^r(E_n)$ [in Russian]. \textit{Imbedding theorems, Matem. Sb.} {\bf 60}(120) (1963), 325--353.

\bibitem {Litt}   W. Littman,  $L^p$ - $L^q$ - estimates for singular integral operators arising from hyperbolic equations,   \textit{Partial differential equations (Proc. Sympos. Pure Math., Vol. XXIII, Univ. California, Berkeley, Calif., 1971)}, pp. 479–-481, Providence, R.I.:  Amer. Math. Soc., 1973.

\bibitem {Marc} A. Marchaud, Sur les deriv\'ees et sur les diff\'erences des fonctions de variables reeles, \textit{Journ. Mathem. Pures et Appl.} \textbf{6} (1927), 337--425.

%\bibitem  {NR} Narayanan, E. K. (6-IIS); Rakesh (1-DE)
%Spherical means with centers on a hyperplane in even dimensions.
%Inverse Problems 26 (2010), no. 3, 035014, 12 pp.

\bibitem  {Nat} F. Natterer. \textit{The Mathematics of computerized tomography}. SIAM, Philadelphia, 2001.

\bibitem  {OS} D.M. Oberlin and E.M. Stein,  Mapping properties of the Radon transform, \textit{Indiana Univ. Math. J.} \textbf{31} (1982), 641--650.

\bibitem{OP}
G. {\'O}lafsson and A. Pasquale, The {${\rm Cos}^\lambda$} and {${\rm Sin}^\lambda$} transforms
  as intertwining operators between generalized principal series
  representations of ${\rm SL}(n+1,\mathbb K)$, \textit{Adv. Math.} \textbf{229}  (2012),
 267--293.


\bibitem  {OPR} G. {\'O}lafsson, A. Pasquale, and B. Rubin,  Analytic and group-theoretic aspects of the cosine transform, {\it Contemp. Math.}, {\bf 598} (2013),  167--188.


\bibitem  {OR} E. Ournycheva  and B. Rubin, Semyanistyi's  integrals and  Radon transforms on
matrix spaces,  {J. Fourier Anal. Appl.} {\bf 14} (2008),  60--88.

\bibitem  {Ru88}  B. Rubin,  One-sided potentials, the spaces $L_{p,r}^\a$ and the inversion of Riesz and Bessel potentials in the
half-space, {\it Math. Nachr.}  {\bf 136} (1988),   177--208.

\bibitem  {Ru89a}  \bysame, Difference regularization of operators of
potential type in $L_p$-spaces (in Russian) {\it Math. Nachr.} {\bf 144} (1989),
119--147.

\bibitem  {Ru89} \bysame,  Multiplier operators connected with the Cauchy problem for the wave equation. Difference regularization (in Russian) {\it Mat. Sb.} {\bf 180} (1989), no. 11, 1524--1547, 1584; translation in {\it Math. USSR-Sb.}  {\bf 68} (1991), no. 2, 391--416.

 \bibitem  {Ru96}   \bysame, {\it Fractional integrals and potentials}, Longman, Harlow, 1996.

\bibitem  {Ru12}  \bysame, The  Radon transform on the Heisenberg group and the transversal Radon transform,    {\it J. of Funct. Analysis}, {\bf 262} (2012), 234--272.

\bibitem  {Ru13}  \bysame,  Semyanistyi fractional integrals and  Radon transforms, {\it Contemp. Math.} {\bf 598} (2013), 221--237.

\bibitem  {Ru13a} \bysame, Funk, Cosine, and Sine transforms on Stiefel and Grassmann manifolds, 	 {\it  J. Geom. Anal.} {\bf 23} (2013),  1441-1497.

\bibitem  {Ru15} \bysame, {\it Introduction to  Radon transforms: With elements of fractional calculus  and harmonic analysis}, Cambridge University Press, 2015.

\bibitem  {Ru22} \bysame,  A note on   the sonar transform and related Radon transforms, Preprint 2022,   	arXiv:2206.05854 [math.FA].

\bibitem  {Sam77}  S.G. Samko,  Test functions vanishing on a given set, and division by a function, {\it Mat. Zametki} (5) {\bf 21} (1977), 677--689.

\bibitem  {Sam}   \bysame, \textit{Hypersingular integrals and their applications}, Taylor \& Francis, Series: Analytical Methods and Special Functions, Volume 5, 2002.

\bibitem {SKM}  S.G. Samko, A.A. Kilbas, and O.I. Marichev, \textit{Fractional integrals and derivatives. Theory and applications.}
   Gordon and Breach Sc. Publ., New York, 1993.

\bibitem  {Sem} V.I. Semyanistyi, On some integral transformations in Euclidean space (in Russian), \textit{Dokl. Akad. Nauk SSSR,}   \textbf{134} (1960),  536--539.

\bibitem  {Sog} Chr. D. Sogge, {\it Fourier integrals in classical analysis}, Cambridge Univ.  Press, 1993.

\bibitem  {Sog1} \bysame, \textit{Lectures on non-linear wave equations}, Second edition. International Press, Boston, MA, 2008.

\bibitem  {Ste56} E.M. Stein, Interpolation of linear operators, {\it Trans. Amer. Math. Soc.} {\bf 83} (1956), 482–-492.

\bibitem  {Ste} \bysame, \textit{Singular integrals and differentiability properties of functions.} Princeton Univ. Press, Princeton, NJ, 1970.

\bibitem {SW}  E.M. Stein and  G. Weiss, \textit{Introduction to Fourier analysis on Euclidean spaces.} Princeton Univ. Press, Princeton, NJ, 1971.

\bibitem {Str81} R.S. Strichartz,   $L^p$-estimates for Radon transforms in Euclidean  and non-euclidean spaces.  \textit{Duke Math. J.} \textbf{ 48}
(1981),  699--727.

\bibitem  {Str}  \bysame, $L^p$ harmonic analysis and Radon transforms on the Heisenberg group, {\it J. Funct. Anal.} {\bf 96} (1991)
350-–406.

\bibitem  {Tao} T. Tao, {\it Stein’s interpolation theorem},\\
 https://terrytao.wordpress.com/2011/05/03/steins-interpolation-theorem/, 3 May 2011.


\bibitem  {Vl} V.S. Vladimirov, \textit{Methods of the theory of generalized functions},  Taylor \& Francis, London and New York, 2003.

\bibitem  {Yos} K. Yoshinaga, On Liouville's differentiation, {\it Bull Kyushu Inst. Technol. Math Natur sci.} {\bf 11} (1964), 1--17.

\end{thebibliography}
\end{document}